\documentclass[10pt]{article}
\usepackage{bbm}
\usepackage{amsfonts}
\usepackage{amsmath}
\allowdisplaybreaks
\usepackage{amssymb}
\usepackage{fancyhdr}
\usepackage{setspace}
\usepackage{latexsym}
\usepackage{mathrsfs}
\usepackage{amsthm}

\usepackage{cite,enumitem,graphicx}
\usepackage[backref,colorlinks,linkcolor=red,anchorcolor=green,citecolor=blue]{hyperref}
\usepackage{cleveref}
\usepackage[english]{babel}

\numberwithin{equation}{section}
\newtheorem{theorem}{Theorem}[section]
\newtheorem{proposition}[theorem]{Proposition}
\newtheorem{lemma}[theorem]{Lemma}
\newtheorem{corollary}[theorem]{Corollary}
\newtheorem{definition}[theorem]{Definition}
\newtheorem{remark}[theorem]{Remark}
\newtheorem{example}[theorem]{Example}

\newcommand{\bb}{\beta}

\newcommand{\bt}{\begin{theorem}}
\newcommand{\et}{\end{theorem}}
\newcommand{\bl}{\begin{lemma}}
\newcommand{\el}{\end{lemma}}
\newcommand{\bd}{\begin{definition}}
\newcommand{\ed}{\end{definition}}
\newcommand{\bc}{\begin{corollary}}
\newcommand{\ec}{\end{corollary}}
\newcommand{\bp}{\begin{proof}}
\newcommand{\ep}{\end{proof}}
\newcommand{\bx}{\begin{example}}
\newcommand{\ex}{\end{example}}
\newcommand{\bi}{\begin{exercise}}
\newcommand{\ei}{\end{exercise}}
\newcommand{\bo}{\begin{proposition}}
\newcommand{\eo}{\end{proposition}}
\newcommand{\br}{\begin{remark}}
\newcommand{\er}{\end{remark}}
\newcommand{\be}{\begin{equation}}
\newcommand{\ee}{\end{equation}}
\newcommand{\ba}{\begin{align}}
\newcommand{\ea}{\end{align}}
\newcommand{\bn}{\begin{enumerate}}
\newcommand{\en}{\end{enumerate}}
\newcommand{\bg}{\begin{align*}}
\newcommand{\bcs}{\begin{cases}}
\newcommand{\ecs}{\end{cases}}

\newcommand{\NN}{{\mathbb N}}

\newcommand{\bean}{\begin{eqnarray*}}
\newcommand{\eean}{\end{eqnarray*}}

\makeatletter 
\renewcommand\theequation{\thesection.\arabic{equation}}
\@addtoreset{equation}{section}
\makeatother
\numberwithin{equation}{section}

\begin{document}

\begin{center}
\textbf{Existence and Multiplicity of Solutions for Fractional $p$-Laplacian Equation Involving Critical Concave-convex Nonlinearities}\\
\end{center}

\begin{center}
Weimin Zhang\\
School of Mathematical Sciences,  Key Laboratory of Mathematics and Engineering Applications (Ministry of Education) \& Shanghai Key Laboratory of PMMP,  East China Normal University, Shanghai 200241, China
\end{center}
\begin{center}
\renewcommand{\theequation}{\arabic{section}.\arabic{equation}}
\numberwithin{equation}{section}
\footnote[0]{\hspace*{-7.4mm}
{E-mail addresses: zhangweimin2021@gmail.com (W. Zhang).}}
\end{center}
\begin{abstract}
We investigate the following fractional $p$-Laplacian equation
\[
\begin{cases}
\begin{aligned}
(-\Delta)_p^s u&=\lambda |u|^{q-2}u+|u|^{p_s^*-2}u  &&\text{in}~\Omega,\\
u&=0~~~~~~~~~~~~~~~~~~~~~&&\text{in}~ \mathbb{R}^n\setminus\Omega,
\end{aligned}
\end{cases}
\]
where $s\in (0,1)$, $p>q>1$, $n>sp$, $\lambda>0$, $p_s^*=\frac{np}{n-sp}$ and $\Omega$ is a bounded domain (with $C^{1, 1}$ boundary). Firstly, we get a dichotomy result for the existence of positive solution with respect to $\lambda$. For $p\ge 2$, $p-1<q<p$, $n>\frac{sp(q+1)}{q+1-p}$, we provide two positive solutions for small $\lambda$. Finally, without sign constraint, for $\lambda$ sufficiently small, we show the existence of infinitely many solutions.
\end{abstract}

\section{Introduction}\label{s1}
In this paper, we are interested in the following fractional $p$-Laplacian equation
\[
\leqno (P_{\lambda})~~~~~~~~~~~~~~~~~~
\begin{cases}
\begin{aligned}
(-\Delta)_p^s u&=\lambda |u|^{q-2}u+|u|^{p_s^*-2}u  &&\text{in}~\Omega,\\
u&=0~~~~~~~~~~~~~~~~~~~~~&&\text{in}~ \mathbb{R}^n\setminus\Omega,
\end{aligned}
\end{cases}
\]
where $s\in (0,1)$, $p>q>1$, $n>sp$, $\lambda>0$, $p_s^*=\frac{np}{n-sp}$ is called the critical Sobolev exponent, and $\Omega \subset {\mathbb R}^n$ is a bounded domain (with $C^{1, 1}$ boundary). $(-\Delta)_p^s$ denotes the fractional $p$-Laplacian operator, and when $u$ is sufficiently smooth, it can be represented pointwisely by
\begin{equation*}
(-\Delta)_p^su(x)=2 ~\underset{\varepsilon \rightarrow 0^+}{\lim}\int_{\mathbb{R}^n\setminus B_{\varepsilon}(x)}{\frac{|u(x)-u(y)|^{p-2}(u(x)-u(y))}{|x-y|^{n+sp}}}dy,
\end{equation*}
up to a normalization constant depending on $n$ and $s$, which is consistent with the usual linear fractional Laplacian $(-\Delta)^s$ when $p=2$, see  \cite{Di12}. Let $f: \Omega\times\mathbb{R}\to \mathbb{R}$ be a Carath\'{e}odory mapping, consider the general fractional $p$-Laplacian equation
\begin{equation}\label{2209241450}
\begin{cases}
\begin{aligned}
(-\Delta)_p^s u&=f (x,u)~~\,~\text{in}~\Omega,\\
u&=0~~~~~~~~~~~\text{in}~ \mathbb{R}^n\setminus\Omega.
\end{aligned}
\end{cases}
\end{equation}
To give a weak formulation of \eqref{2209241450}, we denote the Gagliardo seminorm by
\[
[u]_{s,p}:=\left( \int_{\mathbb{R}^{2n}}{\frac{|u(x)-u(y)|^p}{|x-y|^{n+sp}}}dxdy\right)^{1/p}.
\]
Let
\[
W^{s,p}(\mathbb{R}^n):=\{u\in L^p(\mathbb{R}^n): [u]_{s,p}<\infty\}
\]
be endowed with the norm
\[
\|u\|_{W^{s,p}}:=(|u|_p^p+[u]_{s,p}^p)^{1/p},
\]
where $|\cdot|_{p}$ denotes the usual norm of $L^{p}(\mathbb{R}^n)$. Denote the subspace
\[
W_0^{s,p}(\Omega):=\left\{u\in W^{s,p}(\mathbb{R}^n):  u=0 \; \text{a.e.~in}~\mathbb{R}^n \setminus \Omega\right\},
\]
equivalently renormed with $\|u\|=[u]_{s,p}$ (see \cite[Theorem 7.1]{Di12}), it is well known that $W_0^{s,p}(\Omega)$ is a uniformly convex Banach space. Furthermore, the embedding $W_0^{s,p}(\Omega)\hookrightarrow L^r(\Omega)$ is continuous for $r\in [1,p_s^*]$ and compact for $r\in [1,p_s^*)$, see \cite[Theorems 6.5, 7.1]{Di12}. $(-\Delta)_p^s$ can be variationally regarded as an operator from $W_0^{s,p}(\Omega)$ into its dual space $W_0^{s, p}(\Omega)^*$ as follows,
\[
\langle (-\Delta)_p^s u, v \rangle=\int_{\mathbb{R}^{2n}}{\frac{J_u(x,y)(v(x)-v(y))}{|x-y|^{n+sp}}}dxdy, \quad \forall\; v\in W_0^{s,p}(\Omega),
\]
where $J_u(x,y)=|u(x)-u(y)|^{p-2}(u(x)-u(y))$.\par
 We call $u\in W_0^{s,p}(\Omega)$ a weak solution (respectively weak subsolution, weak supersolution) of \eqref{2209241450} if $f(x, u)\in W_0^{s, p}(\Omega)^*$ and
\begin{equation*}
\langle (-\Delta)_p^s u, v \rangle=(\text{respectively}~\le,\, \ge) \int_{\Omega} f(x, u)vdx, \quad \forall\; v \in W_0^{s,p}(\Omega), v\ge 0.
\end{equation*}
For simplicity, from now on, we omit the term {\sl weak} in the rest of our paper. If $f$ satisfies the growth condition: 
\begin{equation}\label{2209241459}
|f(x,t)|\le C_0 (1+|t|^{r-1}) \quad \mbox{for }\; C_0>0, 1<r\le p_s^*
\end{equation}
and a.e. $x\in\Omega$, $t\in\mathbb{R}$. Thus solutions of problem \eqref{2209241450} coincide with  critical points of the $C^1$ functional
\begin{equation}\label{2307061350}
E(u)=\frac1{p} \|u\|^p-\int_{\Omega}\int_0^{u}f(x, t) dt dx,\quad u\in W_0^{1,p}(\Omega).
\end{equation}
In order to find critical points of $E$, the eigenvalues of $(-\Delta)_p^s$ based on $\mathbb{Z}_2$-cohomological index introduced in \cite{IS2014} are usually used to carry out some linking constructions, see for instance \cite{Perera2015, Iannizzotto016, Mosconi}.
When $f$ satisfies the subcritical growth condition (i.e.~\eqref{2209241459} with $1<r<p_s^*$), many results about existence and multiplicity of solutions of \eqref{2209241450} have been established, see for example \cite{GS2015-1,GS2015, Iannizzotto016}.

In the special case $p=2$, Ros-Oton and Serra \cite{RS2014-1} gave a Pohozaev identity for solutions of \eqref{2209241450}, and when $p\neq 2$, a similar identity for $u\in W_0^{s,p}(\Omega)$ solution to \eqref{2209241450} was conjectured in \cite[Section 7]{Iannizzotto016}. It results that if $f(x, t)=|t|^{p_s^*-2}t$, there may not exist non trivial solutions of \eqref{2209241450} when $\Omega$ is star-shaped. This motivates people to consider the perturbation problem
\begin{equation}\label{2301011916}
\begin{cases}
\begin{aligned}
(-\Delta)_p^s u&=\lambda g(x,u)+|u|^{p_s^*-2}u &&\text{in}~\Omega,\\
u&=0~~~~~~~~~~~~~~~~~~~~&&\text{in}~ \mathbb{R}^n\setminus\Omega,
\end{aligned}
\end{cases}
\end{equation}
where $g: \Omega\times \mathbb{R}\to \mathbb{R}$ satisfies the subcritical growth  \eqref{2209241459} with $1<r<p_s^*$. This type of problems bring new difficulties due to the fact that the associated energy of \eqref{2301011916} cannot satisfy Palais-Smale condition globally because the embedding $W_0^{s,p}(\Omega)\subset L^{p_s^*}(\Omega)$ is not compact. However, the Palais-Smale condition can hold true in suitable thresholds related to the best Sobolev constant
\begin{equation*}
S_{s,p}:=\underset{u\in D^{s,p}(\mathbb{R}^n) \backslash \{ 0\}}{\text{inf}}\frac{[u]_{s,p}^p}{|u|_{p_s^*}^{p}},
\end{equation*}
where $
D^{s,p}(\mathbb{R}^n):=\big\{u\in L^{p^*_s}(\mathbb{R}^n): [u]_{s,p}< \infty\big\}.$

\begin{itemize}
\item Mahwin and Molica Bisci \cite{MM2017} proved that there exists an interval $\mathcal{V}\subset (0,\infty)$ such that for every $\lambda\in \mathcal{V}$, \eqref{2301011916} admits at least one solution, which is a local minimizer of the corresponding energy. The main strategy in \cite{MM2017} is to check the sequentially weakly lower semicontinuity of the functional
\[
H(u)= \frac1p \|u\|^p-\frac{1}{p^*}|u|_{p_s^*}^{p_s^*}
\]
restricted in a sufficiently small ball of $W_0^{s,p}(\Omega)$. As a special case, they showed in \cite[Theorem 1.2]{MM2017} that 
\begin{equation}\label{2305102124}
\begin{cases}
\begin{aligned}
(-\Delta)_p^s u&=\lambda (|u|^{q-2}u+|u|^{r-2}u)+|u|^{p_s^*-2}u && \text{in}~\Omega,\\
u&=0~~~~~~~~~~~~~~~~~~~~~~~~~~~~~&& \text{in}~ \mathbb{R}^n\setminus\Omega,
\end{aligned}
\end{cases}
\end{equation}
has a positive solution for $\lambda\in \mathcal{V}$, provided $2\le q<p<r<p_s^*$.
\item In \cite{Mosconi}, Mosconi {\it et al.}~proved that when $g(x, u)=|u|^{p-2}u$, \eqref{2301011916} has a non trivial solution in the following cases:
\begin{itemize}
\item[$\mathrm{(i)}$] $n = sp^2$ and $\lambda\in (0,\lambda_{1})$;
\item[$\mathrm{(ii)}$] $n>sp^2$ and $\lambda \not\in \{\lambda_{k}\}$;
\item[$\mathrm{(iii)}$] $\frac{n^2}{n+s}>sp^2$;
\item[$\mathrm{(iv)}$] $\frac{n^3+s^3p^3}{n(n+3)}>sp^2$ and $\partial \Omega \in C^{1,1}$,
\end{itemize}
where $\lambda_{k}$ is the $k$-th eigenvalue of $(-\Delta)_p^s$ given in \cite{IS2014}.
\item Bhakta and Mukherjee \cite{BM2017} obtained that when $p\ge 2$, there exist $\lambda_0>0$, $n_0\in\mathbb{N}$ and $q_0\in (1, p)$ such that for all $\lambda\in (0, \lambda_0)$, $n>n_0$ and $q\in (q_0, p)$, $(P_\lambda)$ has at least one sign changing solution.
 \end{itemize}
For the convex-concave nonlinearities, there are also some literature concerning problem $(P_\lambda)$ with the classical $p$-Laplacian operator, i.e. $s=1$,
\begin{equation}\label{22081016}
\begin{cases}
\begin{aligned}
-\Delta_p u&=\lambda |u|^{q-2}u+ |u|^{p^*-2}u \quad\text{in}~\Omega,\\
u&=0~~~~~~~~~~~~~~~~~~~~~~~\,~~~\text{on}~ \partial\Omega,
\end{aligned}
\end{cases}
\end{equation}
where $p^*=\frac{np}{n-p}$ and  $\Delta_p u= \text{div} (|\nabla u|^{p-2}\nabla u)$.   For example:
\begin{itemize}
\item Ambrosetti, Brezis and Cerami  \cite{Ambrosetti94} did seminal works for $1<q<p=2$, and they proved that there exists $\Lambda>0$ such that \par
\begin{itemize}
\item[(i)] problem \eqref{22081016} has at least two positive solutions if $0<\lambda<\Lambda$;\par
\item[(ii)] problem \eqref{22081016} has at least one positive solution if $\lambda=\Lambda$;\par
\item[(iii)] problem \eqref{22081016} has no positive solution if $\lambda>\Lambda$.
\end{itemize}
\item Garc\'{\i}a Azorero, Manfredi and Peral Alonso \cite{GMP} (see also \cite{GP}) generalized the above results for general $p > 1$, provided either $\frac{2n}{n+2}<p<3$, $1<q<p$ or $p\ge 3$, $p>q>\frac{p^*-2}{p-1}$.
\end{itemize}
For the linear fractional Laplacian case, that is $p=2$, $s\in (0,1)$, $n>2s$, $1<q<2$, {\color{red} the above (i)-(iii) for $(P_\lambda)$} were obtained by Barrios, Colorado, Servadei and Soria \cite{BCSS2015}.

\medskip
Motivated by the works mentioned above, we are concerned here with the existence and multiplicity problem for $(P_\lambda)$. To state our results, we introduce some definitions. For all $x\in\Omega$, let
\[
\text{d}_{\Omega}(x):=\text{dist}(x, \partial\Omega),
\]
and consider the weighted space
\[
\mathcal{C}_{s}^0(\overline{\Omega}):=\big\{u\in C^0(\overline{\Omega}): \frac{u}{\text{d}_{\Omega}^s}~\text{admits~a~continuous~extension~to}~\overline{\Omega}\big\}
\]
endowed with the norm $\|u\|_{\mathcal{C}_{s}^0(\overline{\Omega})}=\|\frac{u}{\text{d}_{\Omega}^s}\|_{\infty}$. Our first result is a dichotomy claim which extends the results in \cite{BCSS2015} with $(-\Delta)^s$ and \cite{GMP} with $-\Delta_p$.
\begin{theorem}\label{thm1.1}
Let $s\in (0,1)$, $p\in(1,\infty)$, $q\in(1,p)$, $n>sp$, let $\Omega$ be a bounded domain with $C^{1,1}$ boundary. Then there exists $0<\Lambda<\infty$ such that
\begin{itemize}
\item[(i)] $(P_\lambda)$ has no positive solutions for $\lambda>\Lambda$;
\item[(ii)] $(P_\lambda)$ has a minimal positive solution $u_\lambda$ for any $0<\lambda<\Lambda$; moreover, this family of minimal solutions is increasing with respect to $\lambda$, and $\|u_\lambda\|_{\mathcal{C}_s^0(\overline{\Omega})}\to 0$ as $\lambda\to 0^+$;
\item[(iii)] $(P_\lambda)$ has at least one positive solution $u_\Lambda$ for $\lambda=\Lambda$, given by the pointwise limit of $u_\lambda$ as $\lambda\to\Lambda^-$.
\end{itemize}
\end{theorem}
The authors in \cite{Ambrosetti94} used supersolution and subsolution method to find a minimal solution $v_\lambda$ for $\lambda\in (0, \Lambda)$, which is a stable solution (i.e., the second variation of the energy functional at $v_\lambda$ is nonnegative). The stability yields that $\{v_\lambda\}_{0<\lambda<\Lambda}$ is bounded in $H_0^1(\Omega)$, so the weak limit of $v_\lambda$ is a weak solution for \eqref{22081016} with $\lambda=\Lambda$.  

In contrast to \cite{Ambrosetti94}, we cannot apply Picone's identity due to the presence of the nonlocal term, and there is no related stability theory for fractional $p$-Laplacian so far. Due to the nonlinearity of $(-\Delta)_p^s$, we cannot derive the strong comparison principle directly, that is
$$(-\Delta)_p^su_1\le  (-\Delta)_p^su_2\;\; \text{and}\;\; u_1\neq u_2\Rightarrow u_1<u_2.$$
In \cite{Jarohs2018}, the author derived a version of strong comparison principle, but with rather restrictive assumptions. Our key observation is that all positive solutions to \eqref{2209241450} can be controlled by $\text{d}^s_{\Omega}(x)$, see Proposition \ref{Comp_Le}. Our strategy is to use the comparison principle to get the uniqueness of solution to
\[\leqno (Q_\lambda)~~~~~~~~~~~~~~~~~~~~~~
\begin{cases}
\begin{aligned}
(-\Delta)_p^s u&=\lambda u^{q-1} ~\quad &&\text{in}~\Omega,\\
u&>0 &&\text{in}~ \Omega,\\
u&=0 && \text{in}~ \mathbb{R}^n\setminus\Omega.
\end{aligned}
\end{cases}
\]
We will use the iteration method to show that the unique solution to $(Q_\lambda)$ is less than any positive solution to $(P_\lambda)$, we prove then the existence of a minimal positive solution $u_\lambda$ of $(P_\lambda)$ for all $\lambda\in (0, \Lambda)$, which is increasing with respect to $\lambda$. For any $\lambda\in (0,\Lambda)$, let $0<\lambda''<\lambda<\lambda'<\Lambda$ and
\begin{equation}\label{2307061300}
\Sigma=\big\{u\in W_0^{1,p}(\Omega)\cap \mathcal{C}_{s}^0(\overline{\Omega}): u_{\lambda''}<u<u_{\lambda'}\big\}.
\end{equation}
 Moreover, we use the variational method to find a local minimum solution ${\widehat u}_{\lambda}\in \Sigma$ with respect to the topology of $\mathcal{C}^{0}_s(\overline{\Omega})$. In Lemma \ref{221229}, we show the boundedness of $\{{\widehat u}_{\lambda}\}_{0<\lambda<\Lambda}$ in $W_0^{s, p}(\Omega)$, which yields the boundedness of the minimal solution $\{{u}_\lambda\}_{0<\lambda<\Lambda}$ in $W_0^{s, p}(\Omega)$.
\begin{remark}
Theorem \ref{thm1.1} remains true for \eqref{2301011916} if $g(x, u)=|u|^{q-2}u+|u|^{r-2}u$ with $1<q<p$, $r\in (1, p_s^*)$. Hence it works in the frame of \cite[Theorem 1.2]{MM2017}.
In \cite{GMP}, it was shown that \eqref{22081016} has an extremal solution $u_\Lambda$ in the distributional sense, but the regularity of $u_\Lambda$ was not mentioned. According to the proof of Theorem \ref{thm1.1}, we can assert that $u_\Lambda$ in \cite[Lemma 6.3]{GMP} belongs to $W_0^{1,p}(\Omega)$.
\end{remark}
Now we will consider the existence of a second positive solution to $(P_\lambda)$.
\begin{theorem}\label{2209252009}
Let $s\in (0,1)$, $p\ge 2$, $p-1<q<p$, $n>\frac{sp(q+1)}{q+1-p}$, and $\Omega$ be a bounded domain with $C^{1,1}$ boundary. There exists $\lambda^*>0$ such that for all $\lambda\in (0, \lambda^*)$, problem $(P_\lambda)$ has at least two positive solutions.
\end{theorem}
For equation \eqref{22081016} with $p=2$, one of main points in \cite{Ambrosetti94} for finding two positive solutions is a result of Brezis and Nirenberg \cite{BN1993} which connects variational and nonvariational methods. Roughly speaking, a local minimizer in $C^1$-topology is also a local minimizer in $W_0^{1,2}(\Omega)$. For $p>1$, the equivalence between $C^1(\overline{\Omega})$ and $W_0^{1,p}(\Omega)$ local minimizers of the energy functional was proven respectively in \cite{GMP} and \cite{GZ2003}. 

However, in fractional operator cases, the space $C^1(\overline{\Omega})$ seems to be not suitable for this aim, but $\mathcal{C}^{0}_s(\overline{\Omega})$ can serve as a suitable substitute. Iannizzotto, Mosconi and Squassina \cite{IMS2020_1} proved that for $p\ge 2$ and a given $f$ satisfying \eqref{2209241459}, a local minimizer of the energy $E$ (see \eqref{2307061350}) in $\mathcal{C}^{0}_s(\overline{\Omega})\cap W_0^{s, p}(\Omega)$ with respect to $\mathcal{C}^{0}_s(\overline{\Omega})$-topology is also a local minimizer in $W_0^{s, p}(\Omega)$, see also Barrios {\it et al.} \cite[Proposition 2.5]{BCSS2015} for the case $p=2$.

In virtue of the maximum principle \cite[Theorem A.1]{BF2014}, positive solutions to $(P_\lambda)$ coincide with nontrivial critical points of the following functional defined on $W_0^{s,p}(\Omega)$
\begin{equation}
\label{tildeI}
{\widetilde I}_{\lambda}(u)=\frac1{p}\|u\|^p-\frac{\lambda}{q}\int_{\Omega}{(u^+)}^q dx-\frac{1}{p_s^*}\int_{\Omega}{(u^+)}^{p_s^*} dx,
\end{equation}
where $u^+ =\max\{u, 0\}$. As mentioned above, $(P_\lambda)$ has a minimal solution $u_\lambda$, and ${\widetilde I}_\lambda$ has a minimizer ${\widehat u}_{\lambda}$ in $\Sigma$ (see \eqref{2307061300}) with respect to $\mathcal{C}_s^0(\overline{\Omega})$-topology, which is also a local minimizer in $W_0^{s,p}(\Omega)$ if $p\ge 2$ by \cite{IMS2020_1}. We can assume $u_\lambda={\widehat u}_{\lambda}$, otherwise, the theorem naturally holds true. Under the assumption that ${\widetilde I}_\lambda$ has only two critical points 0 and ${u}_\lambda$, we will prove in Proposition \ref{2209210928} that ${\widetilde I}_\lambda$ satisfies the Palais-Smale condition for all level
\begin{equation}\label{2307082228}
c<c_{s, p}:={\widetilde I}_\lambda({u}_\lambda)+\frac{s}{n}S_{s,p}^{\frac{n}{sp}}.
\end{equation}
It remains to construct a mountain pass geometry of ${\widetilde I}_\lambda$ around ${u}_\lambda$, and check that the mountain pass level strictly less than $c_{s, p}$.\par
It has been conjectured in \cite{Brasco2016} that all minimizers for $S_{s, p}$ are of the form $cU(|x-x_0|/{\varepsilon})$, where
\[
U(x)=\frac{1}{\big(1+|x|^{\frac{p}{p-1}}\big)^{(n-sp)/p}}, \quad x\in \mathbb{R}^n.
\]
As far as we are aware, this conjecture remains open. However, the asymptotic estimates for all minimizers were established by \cite{Brasco2016}. To estimate the mountain pass level, we will make use of some truncation functions $u_{\varepsilon, \delta}$ constructed by Mosconi {\it et al.} \cite{Mosconi} (see \eqref{2301082313}) and some useful integral estimates of $u_{\varepsilon, \delta}$, see subsection \ref{2307062042} below.

In contrast to \cite{Mosconi}, our mountain pass geometry is around ${u}_\lambda$, instead of 0, for which the estimates will be more complex, the nonlocal integral also brings new difficulties to our computations. To overcome these obstacles, we will proceed simultaneously the construction of mountain pass geometry and the estimate for mountain pass level.

To be more precise, we consider $\eta_\delta {u}_\lambda$ where $\eta_\delta$ is a cut-off function (see \eqref{2307082048}). Lemma \ref{2210121045} leads to $\eta_\delta {u}_\lambda\to {u}_\lambda$ in $W_0^{s,p}(\Omega)$ when $\delta\to 0$. We choose ${u}_\lambda$ to be the starting point of mountain pass path, with the terminal point
\begin{equation}\label{23010700}
e=\eta_{\delta}{u}_\lambda+t_0u_{\varepsilon, \delta},
\end{equation}
where $t_0$ depending on $\varepsilon$ and $\delta$ is a positive number such that ${\widetilde I}_\lambda(e)<{\widetilde I}_\lambda({u}_\lambda)$. Consider the set of mountain pass paths
\begin{equation}\label{2307142204}
\Gamma_{\varepsilon,\delta}:=\{\gamma\in C\left([0,1], W_0^{s,p}(\Omega)\right): \gamma(0)=u_\lambda,\, \gamma(1)=e\},
\end{equation}
and the mountain pass level
\begin{equation}\label{2209150938}
m_{\varepsilon,\delta}:=\underset{\gamma\in\Gamma_{\varepsilon,\delta}}{\inf}\,\underset{t\in [0,1]}{\max}{\widetilde I}_{\lambda}(\gamma(t)).
\end{equation}

We select a special mountain pass path
\begin{equation}\label{2301091202}
\gamma_{\varepsilon,\delta}(t) =
\begin{cases}
\begin{aligned}
& \eta_{{2t\delta}}u_{\lambda}\quad&&\text{if}~~0\le t\le \frac12,\\
& \eta_{{\delta}}u_{\lambda}+(2t-1)t_0u_{\varepsilon, \delta}&&\text{if}~~\frac12 < t\le 1.\\
\end{aligned}
\end{cases}
\end{equation}
and check that ${\widetilde I}_\lambda(\gamma_{{\varepsilon,\delta}}(t))$ tends to ${\widetilde I}_\lambda(u_{\lambda})<c_{s, p}$ as $\delta\to 0$ for all $0\le t\le \frac12$. To reach $m_{\varepsilon,\delta}<c_{s, p}$, it suffices to prove
\begin{equation}\label{2307070041}
\underset{t\ge 0}{\sup}\, {\widetilde I}_\lambda(\eta_{{\delta}}u_{\lambda}+tu_{\varepsilon, \delta})< c_{s, p}
\end{equation}
for some $\varepsilon,\delta>0$. In Lemma \ref{2209212151}, for  $p\ge 2$, $p-1<q<p$ and $n>\frac{sp(q+1)}{q+1-p}$, by taking $\varepsilon=\delta^{k+1}$ with suitable choice of $k \in (0, p-1)$, we can claim that \eqref{2307070041} holds whenever $\delta$ is sufficiently small.

A key point is that $u_{\varepsilon, \delta}$ and $\eta_\delta {u}_\lambda$ will be chosen to have disjoint support domains, which permits us to handle many integral estimates. The mountain pass technique provides then a second critical point. Finally, we prove that the mountain pass level is positive for $\lambda$ positive but small, which guarantees the non-triviality of this critical point.
\begin{remark}
In Theorem \ref{2209252009}, it is clear that $\lambda^*\le \Lambda$, that given in Theorem \ref{thm1.1}. However, we don't know how to rule out the triviality of the critical point when the mountain pass level is just zero, hence we cannot claim actually $\lambda^*=\Lambda$.
\end{remark}
Finally, without sign constraint, we obtain the existence of infinitely many solutions for $(P_\lambda)$. Denote
\begin{equation}\label{2304132001}
I_\lambda(u)=\frac1{p}\|u\|^p-\frac{\lambda}{q}\int_{\Omega}{|u|}^q dx-\frac{1}{p_s^*}\int_{\Omega}{|u|}^{p_s^*} dx,\quad u\in W_0^{s,p}(\Omega).
\end{equation}
Notice that actually all the non-zero critical points of ${\widetilde I}_{\lambda}$ are the positive critical points of $I_\lambda$.
\begin{theorem}\label{2209252055}
Let $s\in (0,1)$, $p>1$, $q\in (1, p)$, $n>sp$, and $\Omega$ be a bounded domain. There exists $\lambda^{**}>0$ such that for all $\lambda\in(0, \lambda^{**})$, $(P_\lambda)$ has a sequence of solutions $\{u_j\}$ satisfying $I_\lambda(u_j)\to 0^-$.
\end{theorem}
Garc\'{\i}a Azorero and Peral Alonso in \cite[Theorem 4.5]{Garcia91} proved that for $1<q<p$ and $n>p$, \eqref{22081016} has infinitely many solutions provided $\lambda$ is small (see also \cite{Ambrosetti94} for $p=2$). They used $\mathbb{Z}_2$-genus and Lusternik-Schnirelman theory (see for instance \cite[Theorem 10.9]{AM2007}). Using a dual Fountain theorem due to Bartsch and Willem \cite{BW1995}, when $p=2$, it was showed in \cite[Theorem 3.22]{Willem96} that for $\lambda > 0$ small enough, there exists a sequence of solutions to \eqref{22081016}, whose energies are negative and tends to 0.

Different from the proof of  \cite[Theorem 3.22]{Willem96}, we use the $\mathbb{Z}_2$-genus to construct a sequence of minimax levels $b_j$ for the functional $I_\lambda$ in a small ball $B_r(0)$, see \eqref{23051421} below. We will show that $b_j$ are negative critical values of $I_\lambda$, and lower bounded by $\widetilde{b}_j$ defined in \eqref{2307072358}. By means of weak convergence argument (see Lemma \ref{2306292329}), we will prove that $\widetilde{b}_j\to 0^-$ for small $r$, so does $b_j$.

\begin{remark}
In Section \ref{2306231256}, we provide a space decomposition method for the reflexible and separable Banach space. We believe that this can extend the Fountain theorem \cite[Section 3]{Willem96} and the dual Fountain theorem \cite{BW1995} to the general Banach framework.
\end{remark}
\medskip
The paper is organized as follows. In Section \ref{2306231253}, we introduce some notations and preliminary results. The proof of Theorems \ref{thm1.1}, \ref{2209252009} and \ref{2209252055} are completed respectively in Sections \ref{2306231254}-\ref{2306231256}.

\section{Notations and Preliminaries}\label{2306231253}
In this paper, $C, C', C_1, C_2,...$ denote always generic positive constants. $|\cdot|_{p}$ means the usual norm of $L^{p}(\mathbb{R}^n)$ or $L^{p}(\Omega)$. $\|\cdot\|$ denotes the norm of $W_0^{s,p}(\Omega)$. We use $\varphi_1$ to denote the eigenfunction corresponding to the first eigenvalue $\lambda_1$ of $(-\Delta)_p^s$ in $W_0^{s, p}(\Omega)$ such that $\varphi_1>0$ in $\Omega$ and $|\varphi_1|_{\infty}=1$.

\subsection{Truncations for the minimizers of $S_{s, p}$}\label{2307062042}
From \cite{Brasco2016}, we know that if $s\in (0, 1)$, $p>1$, there exists a minimizer $U\in D^{s, p}(\mathbb{R}^n)$ for $S_{s,p}$. Up to scaling and translation, we can assume that $U(0)=1$, $U$ is radially symmetric, nonnegative, radially decreasing and resolves in ${\mathbb R}^n$
\begin{equation}\label{2209100830}
(-\Delta)_p^s U=U^{p_s^*-1}.
\end{equation}
Without confusion, we denote also $U(x) = U(r)$ with $r=|x|$. For any $\varepsilon>0$, let
\begin{equation}\label{2307142307}
U_{\varepsilon}(x)=\frac{1}{\varepsilon^{(n-sp)/p}}U\left(\frac{x}{\varepsilon}\right)
\end{equation}
which is also a minimizer for $S_{s,p}$, and satisfies \eqref{2209100830}. We recall some results shown in Lemmas 2.2, 2.6 and 2.7 of \cite{Mosconi} respectively.

\begin{lemma}\label{2209051345}
There exist constants $c_1, c_2>0$ and $\theta>1$ such that for all $r\ge 1$,
\begin{equation}\label{2209151902}
\frac{c_1}{r^{(n-sp)/(p-1)}}\le U(r) \le \frac{c_2}{r^{(n-sp)/(p-1)}}
\end{equation}
and
\begin{equation}\label{2307150016}
2U(\theta r) \le U(r).
\end{equation}
\end{lemma}

Let $\theta$ be as in Lemma \ref{2209051345}, fix $\eta \in C^{\infty}(\mathbb{R}^n,[0,1])$ such that
\begin{equation}\label{2209151128}
\eta(x)=\left\{
\begin{array}{lll}
0,&\text{if} \, |x|\le 2\theta,\\
1,&\text{if} \, |x|\ge 3\theta,\\
\end{array}
\right.
\end{equation}
and for any $\delta>0$, denote
\begin{equation}\label{2307082048}
\eta_{\delta}(x)=\eta\left(\frac{x}{\delta}\right).
\end{equation}

\begin{lemma}\label{2209151022}
Assume that $0\in \Omega$. Then there exists a constant $C=C(n, \Omega, p, s)>0$ such that for any $v\in W_0^{s,p}(\Omega)$, $\delta>0$ satisfying $(-\Delta)^s_pv\in L^{\infty}(\Omega)$ and $B_{5\theta \delta}(0)\subset \Omega,$ there holds
\[
\lVert v\eta_\delta \rVert^p\le \lVert v \rVert^p + C\left| (-\Delta)_p^sv \right|_{\infty}^{p/(p-1)}\delta^{n-sp}.
\]
\end{lemma}

For any $\varepsilon, \delta>0$, we denote
\begin{equation}\label{2307150015}
M_{\varepsilon, \delta}=\frac{U_{\varepsilon}(\delta)}{U_{\varepsilon}(\delta)-U_{\varepsilon}(\theta \delta)}.
\end{equation}
Let
\[
\begin{aligned}
g_{\varepsilon, \delta}(t)=
\begin{cases}
0, ~~~~~~~~~~~~~~~~~~~~~~~~~~\,~~\text{if}~0\le t\le U_{\varepsilon}(\theta \delta),\\
M_{\varepsilon, \delta}^p(t-U_{\varepsilon}(\theta \delta)), ~~~~~~~~\text{if}~U_{\varepsilon}(\theta \delta)\le t\le U_{\varepsilon}(\delta),\\
t+U_{\varepsilon}(\delta)(M_{\varepsilon, \delta}^{p-1}-1), ~~~\text{if}~t\ge U_{\varepsilon}(\delta),
\end{cases}
\end{aligned}
\]
and
\begin{equation}\label{2307150014}
\begin{aligned}
G_{\varepsilon, \delta}(t)=\int_{0}^t g'_{\varepsilon, \delta}(\tau)^{1/p} d\tau=
\begin{cases}
0, ~~~~~~~~~~~~~~~~~~~~~~~~~~~~~\text{if}~0\le t\le U_{\varepsilon}(\theta \delta),\\
M_{\varepsilon, \delta}\,(t-U_ {\varepsilon}(\theta \delta)), ~\,~~~~~~~\text{if}~U_{\varepsilon}(\theta \delta)\le t\le U_{\varepsilon}(\delta),\\
t, ~~~~~~~~~~~~~~~~~~~~~~~~~~~~~~\text{if}~t\ge U_{\varepsilon}(\delta).
\end{cases}
\end{aligned}
\end{equation}
The functions $g_{\varepsilon, \delta}$ and $G_{\varepsilon, \delta}$ are decreasing and absolutely continuous. Then the radially nonincreasing function
\begin{equation}\label{2301082313}
u_{\varepsilon, \delta}(r)=G_{\varepsilon, \delta}(U_{\varepsilon}(r)) \leq U_{\varepsilon}(r),
\end{equation}
satisfies
\begin{equation}\label{2209151133}
u_{\varepsilon, \delta}(r)= U_{\varepsilon}(r) \text{ if}~r\le \delta \quad \mbox{and}\quad {\rm supp}(u_{\varepsilon, \delta}) \subset B_{\theta\delta}(0).
\end{equation}

\begin{lemma} There exists a constant $C=C(n,p,s)>0$ such that for any $\varepsilon\le \frac{\delta}{2}$,
 \begin{gather}
	\lVert u_{\varepsilon,\delta}\rVert^p \le S_{s,p}^{\frac{n}{sp}}+C\left(\frac{\varepsilon}{\delta}\right)^{(n-sp)/ (p-1)}\label{2209021656},\\
	 |u_{\varepsilon,\delta}|_{p^*_s}^{p_s^*}\ge S_{s,p}^{\frac{n}{sp}}-C{\left(\frac{\varepsilon}{\delta}\right)}^{\frac{n}{p-1}}\label{2209021657},\\
|u_{\varepsilon, \delta}|_p^p\ge
	\begin{cases}
      		 \frac1C\varepsilon^{sp}|\log{\left(\frac{\varepsilon}{\delta}\right)}|,~~\text{if}~~n=sp^2,\\
      		 \frac1C\varepsilon^{sp},~~~~~~~~~~~~\text{if}~~n>sp^2.\\
        \end{cases}
    \end{gather}
\end{lemma}

\subsection{Further useful results}
Consider fractional equation:
\begin{equation}\label{2208122115}
(-\Delta)_p^s u \ge 0 \quad \text{in}~\Omega; \quad u=0\quad \text{in}~\mathbb{R}^n\backslash \Omega.
\end{equation}
Referring to \cite[Theorem A.1]{BF2014}, we recall the following useful maximum principle.
\begin{lemma}\label{22080821}
Let  $s\in (0,1)$, $1<p<\infty$ and $\Omega\subset \mathbb{R}^n$ be an open bounded connected set. If $u\in W_0^{s,p}(\Omega)$ is a nonnegative solution to \eqref{2208122115} and $u\not\equiv 0$ in $\Omega$, then $u>0$ almost everywhere in $\Omega$.
\end{lemma}
The following comparison principle for fractional $p$-Laplacian can be found in \cite[Proposition 2.10]{IMS2016} (see also \cite[Lemma 9]{LL2014}).
\begin{lemma}\label{2307091452}
Let $p>1$ and $\Omega$ be an open bounded set. If $u, v\in W_0^{s,p}(\Omega)$ satisfy that for all $\varphi\in W_0^{s,p}(\Omega)$ with $\varphi\ge 0$ in $\Omega$,
\[
\int_{\mathbb{R}^{2n}}{\frac{J_u(x,y)(\varphi(x)-\varphi(y))}{|x-y|^{n+sp}}}dxdy\ge\int_{\mathbb{R}^{2n}}{\frac{J_v(x,y)(\varphi(x)-\varphi(y))}{|x-y|^{n+sp}}}dxdy,
\]
then $u\ge v$ in $\Omega$.
\end{lemma}
Iannizzotto, Mosconi and Squassina \cite{IMS2020_1} proved the following equivalence.
\begin{lemma}\label{2208251418}
Let $p\ge 2$, $s\in (0,1)$ and $n>sp$. Let $\Omega\subset \mathbb{R}^n$ be a bounded domain with a $C^{1,1}$-boundary, $f: \Omega\times\mathbb{R}\to \mathbb{R}$ be a Carath\'{e}odory mapping satisfying \eqref{2209241459}. Then, for any $u_0\in W_0^{s,p}(\Omega)$, the following are equivalent:
\begin{itemize}
\item[(i)] there exists $\rho>0$ such that $E(u_0+v)\ge E(u_0)$ for all $v\in W_0^{s,p}(\Omega)$, $\|v\|\le \rho$;
\item[(ii)] there exists $\sigma>0$ such that $E(u_0+v)\ge E(u_0)$ for all $v\in W_0^{s,p}(\Omega)\cap \mathcal{C}^{0}_s(\overline{\Omega})$, $\|v\|_{\mathcal{C}^{0}_s(\overline{\Omega})}\le \sigma$,
\end{itemize}
where $E$ is defined in \eqref{2307061350}.
\end{lemma}
Next is a H\"older regularity result for solutions of \eqref{2209241450}, which comes from \cite[Theorem 3.3, Remark 3.4]{CMS2018} and \cite[Theorem 1.1]{IMS2016}.
\begin{proposition}\label{2209252041}
Let $\Omega$ be a bounded domain with $C^{1,1}$ boundary, $f$ satisfy \eqref{2209241459}. Then there exists $\alpha\in (0,s]$ such that if $u\in W_0^{s,p}(\Omega)$ is a solution of \eqref{2209241450}, we have $u\in C^{\alpha}(\overline{\Omega})$.
\end{proposition}

If we assume $f(x, u)\in L^{\infty}(\Omega)$ and $f(x, u)\ge 0$, the solutions of \eqref{2209241450} will satisfy also a kind of Hopf's lemma \cite[Theorem 1.5]{DQ2017}. In particular, the solutions can be controlled by ${\text{d}^s_{\Omega}(x)}$. The decay information of solutions near the boundary is very useful to construct suitable subsolutions or supersolutions. 
\begin{proposition}\label{Comp_Le}
Let $\Omega$ be a bounded domain with $C^{1,1}$ boundary and let $u\in W_0^{s,p}(\Omega)$ be a weak solution of the following equation
\begin{equation}\label{2210122106}
\begin{cases}
\begin{aligned}
(-\Delta)^s_p u&=f ~~~\text{in}~\Omega,\\
u&=0~\,~~\text{in}~ \mathbb{R}^n\setminus\Omega,
\end{aligned}
\end{cases}
\end{equation}
where $f\in L^{\infty}(\Omega)$, $f\ge 0$ and $f\not\equiv 0$. Then there are two positive numbers $C$, $C'$ such that
\[
C{\text{d}^s_{\Omega}(x)}\le {u(x)}\le C'{\text{d}^s_{\Omega}(x)}, \quad \mbox{for a.e.}\; x\in\Omega.
\]
\end{proposition}

\begin{proof}
By Lemmas \ref{22080821} and \ref{2307091452}, we have $u>0$ a.e. in $\Omega$. Since $f\in L^{\infty}(\Omega)$, the H\"{o}lder continuity follows from \cite[Theorem 1.1]{IMS2016}. From the Hopf's lemma \cite[Theorem 1.5]{DQ2017}, there exists $C$ such that $C {\text{d}^s_{\Omega}(x)}\le u(x).$ Morover, it follows from \cite[Theorem 4.4]{IMS2016} that there exists $C'$ such that $C' {\text{d}^s_{\Omega}(x)} \ge u(x)$.
\end{proof}

At last, we summarize the super-subsolution method as follows.
\begin{proposition}\label{Super_Sub_Le}
Let $f: \Omega\times\mathbb{R}\to \mathbb{R}$ be a Carath\'{e}odory mapping satisfying that $f(x, t)$ is continuous and increasing in $t$ for a.e. $x\in\Omega$. Let $\underline{u}\in W_0^{s,p}(\Omega)$ be a subsolution to \eqref{2209241450}, and $\overline{u}\in W_0^{s,p}(\Omega)$ be a supersolution to \eqref{2209241450} such that $\underline{u}\le \overline{u}$ and
\[
f(\cdot, \underline{u}(\cdot)),~f(\cdot, \overline{u}(\cdot))\in L^{\frac{p_s^*}{p_s^*-1}}(\Omega).
\]
Then there exists a weak solution $u\in W_0^{s,p}(\Omega)$ of \eqref{2209241450} such that $\underline{u}\le u \le \overline{u}$.
\end{proposition}
\begin{proof}
Let $u_0:=\underline{u}$, and by induction we define $u_j$ $(j\ge 1)$ as solutions of 
\begin{equation}\label{2208222050}
\begin{cases}
\begin{aligned}
(-\Delta)_p^s u_{j+1}&=f(x, u_j) &&\text{in}~\Omega,\\
u_{j+1}&=0~~~~&&\text{in}~ \mathbb{R}^n\setminus\Omega.\\
\end{aligned}
\end{cases}
\end{equation}
By comparison principle, then $\underline{u}\le u_j \le u_{j+1} \le \overline{u}$. Consequently,
\begin{equation*}
|u_j|\le \max\{|\underline{u}|, |\overline{u}|\}, \quad |f(x, u_j)|\le \max\{|f(x, \underline{u})|, |f(x, \overline{u})|\}.
\end{equation*}
Multiplying \eqref{2208222050} by $u_{j+1}$,
\[
\begin{aligned}
\lVert u_{j+1}\rVert^{p}=&\int_{\Omega} f(x, u_j)u_{j+1} dx & \le  \int_{\mathbb{R}^n}  \max\{|\underline{u}|, |\overline{u}|\}\max\{|f(x, \underline{u})|, |f(x, \overline{u})|\} dx.
\end{aligned}
\]
Therefore $\{u_j\}$ is a bounded sequence in $W_0^{s,p}(\Omega)$, and $u_j$ weakly converges to some $u$ in $W_0^{s,p}(\Omega)$, $u_j(x)\rightarrow u(x)$ a.e. in $\mathbb{R}^n$. Consequently, $u$ is a weak solution to \eqref{2209241450} and $\underline{u}\le u \le \overline{u}$.
\end{proof}

\section{Proof of Theorem \ref{thm1.1}}\label{2306231254}
Let
\begin{equation}\label{2307091941}
\Lambda:=\sup\{\lambda>0: (P_\lambda)~ \text{has a positive solution in }\; W_0^{s, p}(\Omega)\}.
\end{equation}
We first verify that $\Lambda$ is finite and nonzero. For simplicity, we will not repeat always the regularity assumption for $\Omega$ and we also often omit to repeat $n > sp$. 
\begin{lemma}\label{2209270925}
For $s\in (0,1)$, $p>q>1$ and $n>sp$, we have $0<\Lambda<\infty$.
\end{lemma}
\begin{proof}
Let $e$ be the solution of the following equation
\begin{equation}\label{Super_Eq}
\left\{
\begin{split}
(-\Delta)_p^s e&=1 ~~~~~\text{in}~\Omega,\\
e&=0 ~~~~~\text{in}~ \mathbb{R}^n\setminus\Omega.
\end{split}
\right.
\end{equation}
Since $1<q<p<p_s^*$, we can find $\lambda_0, M>0$ small satisfying
\[
\lambda_0 M^{q-p} |e|^{q-1}_{\infty}+M^{p_s^*-p} | e |^{p_s^*-1}_{\infty} \leq 1,
\]
which implies that for $0<\lambda\le\lambda_0$,
\[
\begin{split}
(-\Delta)_p^s (Me)&=M^{p-1}\ge \lambda M^{q-1}| e |^{q-1}_{\infty}+M^{p_s^*-1} | e |^{p_s^*-1}_{\infty}\ge \lambda(Me)^{q-1}+(Me)^{p_s^*-1}.
\end{split}
\]
So $Me$ is a supersolution of $(P_\lambda)$. Fix any $0<\lambda\le\lambda_0$, there exists $\varepsilon_0 \in (0, 1)$ (recall that $\varphi_1\in C(\overline{\Omega})$, $|\varphi_1|_\infty=1$ and $\varphi_1>0$ in $\Omega$), such that
$$\lambda_1 \le \lambda(\varepsilon \varphi_1)^{q-p} \le \lambda(\varepsilon \varphi_1)^{q-p}+(\varepsilon \varphi_1)^{p_s^*-p}, \quad \forall\; 0<\varepsilon<\varepsilon_0.$$
We check readily that $\varepsilon \varphi_1$ is a subsolution of $(P_\lambda)$. According to Proposition \ref{Comp_Le}, we can choose $\varepsilon > 0$ small enough such that $\varepsilon \varphi_1\le Me$. By Proposition \ref{Super_Sub_Le},  $(P_\lambda)$ has a positive solution, which yields $\Lambda>0$. \par
From \cite[Theorem 4.1]{BP2016}, we know that $\lambda_1$ is an isolated eigenvalue, so there is $\widetilde{\lambda}>\lambda_1$, which is not an eigenvalue of $(-\Delta)_p^s$. Assume by contradiction that $\Lambda=\infty$, we can choose then $$\widehat \lambda > \max_{t \ge 0} \big(\widetilde{\lambda}t^{p-q} - t^{p_s^*-q}\big)$$ such that $(P_{\widehat{\lambda}})$ has a positive solution $u_{\widehat{\lambda}}$. Hence $u_{\widehat{\lambda}}$ is a supersolution of the following equation
\begin{equation}\label{2209141646}
\left\{
\begin{split}
(-\Delta)_p^s u&=\widetilde{\lambda} |u|^{p-2}u &&\text{in}~\Omega,\\
u&=0~~~~~~~~&& \text{in}~ \mathbb{R}^n\setminus\Omega.\end{split}
\right.
\end{equation}
Applying again Propositions \ref{2209252041}--\ref{Comp_Le}, there is $\varepsilon>0$ small such that $\varepsilon \varphi_1\le u_{\widehat{\lambda}}$. Since $\lambda_1<\widetilde{\lambda}$, $\varepsilon \varphi_1$ is a subsolution to \eqref{2209141646}. By Proposition \ref{Super_Sub_Le}, there exists a positive solution of \eqref{2209141646}, hence a contradiction occurs with the choice of $\widetilde \lambda$. It means that $0<\Lambda<\infty$.
\end{proof}


In the sequel, we are going to find the minimal positive solution of $(P_\lambda)$. We claim that the minimal solution can be derived by iterating from the solution of $(Q_\lambda)$. The first step is to prove that the solution of $(Q_\lambda)$ is unique.
\begin{lemma}\label{Unique_Le}
Let $s\in (0,1)$, $p>1$, $n\ge 1$, $q\in(1,p)$ and $\lambda>0$. Then there exists a unique solution $v_\lambda$ to $(Q_\lambda)$. Moreover, for any $\lambda>0$, $v_\lambda = \lambda^\frac{1}{p-q}v_1$.
\end{lemma}
\begin{proof}
Consider the functional
\[
E_\lambda(v)=\frac\lambda p\rVert v \lVert^p-\frac1q |v^+|_q^q,\quad v\in W_0^{s,p}(\Omega),
\]
which is coercive and weakly lower semi-continuous. Therefore, there is a global minimizer $v_\lambda\in W_0^{s,p}(\Omega)$ for $E_{\lambda}$. By $p>q>1$, we obtain $E_\lambda(v_\lambda)<0$ hence $v_\lambda\neq 0$. By maximum principle, we have $v_\lambda> 0$ in $\Omega$. We claim that $v_\lambda$ is the unique positive solution to $(Q_\lambda)$. Let $w$ be another positive solution, using Proposition \ref{2209252041}, $v_\lambda$, $w \in L^{\infty}(\Omega)$, then by Proposition {\ref{Comp_Le}}, there are two positive constants $C$, $C'$ such that
\begin{equation}\label{Unique_Le_Neq}
Cw(x) \le v_\lambda(x) \le C'w(x), \quad \mbox{for a.e. $x\in \Omega$}.
\end{equation}
Let $\beta_0 :=\sup\{\ell: v_\lambda\ge \ell w ~\text{a.e. in} ~\Omega \}$, by (\ref{Unique_Le_Neq}), we have $0<\beta_0 <\infty$. Clearly
\[
(-\Delta)_p^s v_\lambda=\lambda v_\lambda^{q-1}\ge \lambda \beta_0^{q-1} w^{q-1}=(-\Delta)_p^s \big(\beta_0^{\frac{q-1}{p-1}}w\big).
\]
By comparison principle, we get $v_\lambda\ge \beta_0^{\frac{q-1}{p-1}}w$, so that $\beta_0\ge\beta_0^{\frac{q-1}{p-1}}$, hence $\beta_0\ge 1$, which means $v_\lambda \geq w$ in $\Omega$. Similarly, $w \geq v_\lambda$, so we have $w = v_\lambda$. 

The uniqueness yields readily the expression of $v_\lambda$ by $v_1$,  since $\lambda^\frac{1}{p-q}v_1$ resolves $(Q_{\lambda})$.
\end{proof}
Due to the nonlinearity of $(-\Delta)_p^s$, we cannot derive the strong comparison principle directly. Here we will show the strict comparison between minimal solutions, thanks to the following inequality.
\begin{lemma}\label{2208242203}
Let $f_\lambda(t)=\lambda t^{q-1}+t^{p_s^*-1}$ for $\lambda>0$, $t> 0$, $q < p_s^*$.
Then for any $0<\lambda< \lambda'<\infty$ and $M > 0$, there exists $\bb_0>1$ such that $f_{\lambda}(\bb_0t)\le f_{\lambda'}(t)$ for $0< t\le M$.
\end{lemma}
\begin{proof}
Let $1<\bb \le (\frac{\lambda'+\lambda}{2\lambda})^{\frac1{q-1}}$, then
\[
\begin{split}
f_{\lambda'}(t)-f_{\lambda}(\bb t)&=\Big(\frac{\lambda'+\lambda}{2}-\lambda \bb^{q-1}\Big)t^{q-1}+\frac{\lambda'-\lambda}{2}t^{q-1}+(1-\bb^{p_s^*-1})t^{p_s^*-1}\\
&\ge \frac{\lambda'-\lambda}{2}t^{q-1}+(1-\bb^{p_s^*-1})t^{p_s^*-1}.
\end{split}
\]
As $\lambda'>\lambda$ and $p_s^*>q$, there exist $t_0>0$, $\beta_1>1$ such that $f_{\lambda}(\bb_1t)\le f_{\lambda'}(t)$ for all $0<t<t_0$. On the other hand, there exists some $\bb_2 > 1$ but close enough such that $f_{\lambda}(\bb_2t)\le f_{\lambda'}(t)$ for all $t \in [t_0, M]$. $\bb_0 = \min(\bb_1, \bb_2)$ will satisfy the claim.
\end{proof}

Now, we are ready to prove the existence of minimal positive solution to $(P_\lambda)$.
\begin{proposition}\label{Minimal_Le}
For any $0<\lambda<\Lambda$, the problem $(P_\lambda)$ has a minimal positive solution $u_\lambda$ such that $u_\lambda<u_{\lambda'}$ if $0<\lambda<\lambda'<\Lambda$. Moreover, $\|u_\lambda\|_{\mathcal{C}_s^0(\overline{\Omega})}\to 0$ as $\lambda\to 0^+$.
\end{proposition}
\begin{proof}
Let $w_{\lambda}$ be an arbitrary positive solution of $(P_{\lambda})$ with $\lambda \in (0, \Lambda)$. Let $v_\lambda$ be the unique solution of $(Q_{\lambda})$, we claim that $v_\lambda \leq w_\lambda$. Indeed, $w_{\lambda}$ is a supersolution to $(Q_\lambda)$. Using Propositions \ref{2209252041}--\ref{Comp_Le}, we see that for $\varepsilon > 0 $ small enough, there holds $\varepsilon \varphi_1 \leq w_\lambda$ and $\varepsilon \varphi_1$ is a subsolution to $(Q_{\lambda})$. As before, we get a solution of $(Q_\lambda)$ between $w_\lambda$ and $\varepsilon \varphi_1$, there holds $v_\lambda \le w_{{\lambda}}$ seeing Lemma \ref{Unique_Le}. 

Take any $\lambda \in (0, \Lambda)$. By the definition of $\Lambda$, there exists $\lambda' \in (\lambda, \Lambda)$ such that $(P_{\lambda'})$ has a positive solution, denoted by $w_{\lambda'}$. Applying Lemma \ref{Unique_Le} and the above argument, there holds $v_{\lambda}\le v_{\lambda'}\le w_{\lambda'}$. As $v_{\lambda}$ and $w_{\lambda'}$ are respectively sub and supersolution to $(P_\lambda)$, we can use the classical monotone iteration method or Proposition \ref{Super_Sub_Le}, starting with $v_\lambda$, to obtain a positive solution $u_\lambda$ for $(P_{\lambda})$. As $w_{\lambda'}$ can be replaced by arbitrary solution of $(P_\lambda)$ or $(P_{\lambda'})$, we conclude that $u_\lambda$ is the minimal positive solution to $(P_\lambda)$ and $u_\lambda\le u_{\lambda'}$.

More precisely, for $0 < \lambda< \lambda' < \Lambda$, using Lemma \ref{2208242203}, there exists $\bb_0>1$ such that
\[
\begin{split}
(-\Delta)_p^s u_{\lambda'} =f_{\lambda'}(u_{\lambda'}) \ge f_{\lambda'}(u_\lambda) \ge f_\lambda(\bb_0u_\lambda) 
= \lambda \bb_0^{q-1}u_\lambda^{q-1}+ \bb_0^{p_s^*-1}u_{\lambda}^{p_s^*-1}
\ge (-\Delta)_p^s (\beta_0^{\frac{q-1}{p-1}}u_\lambda).
\end{split}
\]
Applying again the comparison principle, we see that $u_{\lambda'}\ge \beta_0^{\frac{q-1}{p-1}} u_\lambda > u_\lambda$ in $\Omega$. 

Finally, coming back to the proof of Lemma \ref{2209270925}, for any $M > 0$ small enough, there exists $\lambda > 0$ small such that $Me$ is a supersolution of $(P_\lambda)$. As above, since $u_\lambda$ is the minimal solution, we get $u_\lambda\le Me$. So $\|u_\lambda\|_{\mathcal{C}_s^0(\overline{\Omega})}\to 0$ as $\lambda\to 0^+$.
\end{proof}

For any $\lambda\in (0,\Lambda)$, we can find $0<\lambda''<\lambda<\lambda'<\Lambda$. By Proposition \ref{Minimal_Le}, there are three minimal solutions corresponding to $\lambda''$, $\lambda$, $\lambda'$, denoted by $u_1, u_\lambda, u_2$ respectively. Let us define for every $x\in \Omega$,
\[
\widehat f_\lambda(x,t) =\left\{
\begin{array}{lll}
f_\lambda(u_2(x)),&\mbox{if }\; t\ge u_2(x);\\
f_\lambda(t),&\mbox{if }\; u_2(x)>t \geq 0.\\
\end{array}
\right.
\]
Consider the following equation
\begin{equation}\label{2208242220}
\begin{cases}
\begin{aligned}
(-\Delta)_p^s u &= \widehat f_\lambda (x,u)&&\mbox{in }\;  \Omega,\\
u& \ge 0 && \mbox{in }\;  \Omega,\\
u& =0 &&\mbox{in }\;  \mathbb{R}^n\setminus\Omega,\\
\end{aligned}
\end{cases}
\end{equation}
with the associated energy functional
\[
\widehat{I}_\lambda (u)=\frac1p \lVert u \rVert^p-\int_{\Omega}{\widehat F}_\lambda (x,u)dx, \quad \mbox{where }\; {\widehat F}_\lambda (x,u)=\int_0^{u^+}\widehat f_\lambda (x,t)dt.
\]
One can prove that $\widehat{I}_\lambda$ is weakly lower continuous and coercive in $W_0^{s,p}(\Omega)$, so $\widehat{I}_\lambda$ can achieve its global minimum at some ${\widehat u}_\lambda\in W_0^{s,p}(\Omega)$ such that $0 < {\widehat u}_\lambda \le u_2$. We can find a subsolution $\varepsilon \varphi_1$ of $(P_{\lambda''})$ such that $\varepsilon \varphi_1\le {\widehat u}_\lambda$ and use super-subsolution method to conclude ${\widehat u}_\lambda\ge u_1$.

In virtue of Proposition \ref{Minimal_Le}, $u_1 <u_2$ in $\Omega$. Let $\Sigma$ be as in \eqref{2307061300}, i.e.
$\Sigma=\{u\in W_0^{1,p}(\Omega)\cap \mathcal{C}_{s}^0(\overline{\Omega}): u_1<u<u_2\}.$
We shall prove that ${\widehat u}_{\lambda}$ is a interior point of $\Sigma$ with respect to $ \mathcal{C}_{s}^0(\overline{\Omega})$-topology. 
\begin{lemma}\label{221229}
${\widehat u}_{\lambda}$ is a solution to $(P_\lambda)$, and ${\widehat u}_{\lambda}\in \Sigma^o$ is a local minimizer of $\widetilde{I}_\lambda$ given in \eqref{tildeI}. In other words, there exists $\sigma>0$ such that for any $h\in W_0^{s,p}(\Omega)\cap \mathcal{C}_{s}^{0}(\overline{\Omega})$ with $\lVert h \rVert_{\mathcal{C}_{s}^{0}(\overline{\Omega})}<\sigma$, 
\[
{\widehat u}_{\lambda}+h\in \Sigma \quad \mbox{and}\quad \widetilde{I}_\lambda({\widehat u}_{\lambda}+h)\ge \widetilde{I}_\lambda({\widehat u}_{\lambda}).
\]
Moreover, we have ${\widetilde I}_\lambda({\widehat u}_\lambda)<0$, 
\begin{equation}\label{2209110032}
\begin{split}
(p-1)\|\widehat{u}_{\lambda}\|^p-(q-1)\lambda|\widehat{u}_{\lambda}|_q^q-({p_s^*}-1)|\widehat{u}_{\lambda}|_{p_s^*}^{p_s^*}\ge 0,
\end{split}
\end{equation}
and there exists $C>0$ such that $\lVert \widehat{u}_\lambda \rVert<C$ for all $0<\lambda<\Lambda$.
\end{lemma}
\begin{proof}
Using Lemma \ref{2208242203}, there exists $\beta >1$ such that
\[
\begin{split}
(-\Delta)_p^s u_2&= f_{\lambda'}(u_2)\ge  f_{\lambda'}({\widehat u}_\lambda) \ge   f_{\lambda}(\beta{\widehat u}_\lambda) = {\lambda} \beta^{q-1}{\widehat u}_{\lambda}^{q-1}+ \beta^{p_s^*-1}{\widehat u}_{\lambda}^{p_s^*-1}\ge (-\Delta)_p^s (\beta^{\frac{q-1}{p-1}}\widehat{u}_\lambda),
\end{split}
\]
which together with the comparison principle implies
\begin{equation}\label{2208251302}
u_2\ge \gamma_2\widehat{u}_\lambda, \quad\text{where }~ \gamma_2:=\beta^{\frac{q-1}{p-1}}>1 .
\end{equation}
Similarly, there exists $\gamma_1>1$ such that
\begin{equation}\label{2208251303}
{\widehat u}_\lambda\ge \gamma_1 u_1.
\end{equation}
Moreover, by Proposition \ref{Comp_Le}, there exists $\beta'>0$ such that
\begin{equation}\label{2307091828}
u_1(x) \geq \beta' \text{d}^s_{\Omega}(x).
\end{equation}
Choose now $\sigma:=\left[\min(\gamma_1, \gamma_2) -1\right]\beta'$. Let $h\in W_0^{s,p}(\Omega)\cap \mathcal{C}_{s}^{0}(\overline{\Omega})$ satisfy $\lVert h \rVert_{\mathcal{C}_{s}^{0}(\overline{\Omega})}<\sigma$, 
then by \eqref{2208251302}--\eqref{2307091828}, there holds
\[
{\widehat u}_\lambda+h >{\widehat u}_\lambda-(\gamma_1-1)\beta'{\text{d}^s_{\Omega}(x)}>{\widehat u}_\lambda-(\gamma_1-1)u_1\ge u_1
\]
and similarly ${\widehat u}_\lambda + h < u_2$, so ${\widehat u}_{\lambda}+h\in \Sigma$. Furthermore,
\[
\begin{aligned}
{\widetilde I}_\lambda ({\widehat u}_\lambda+h)&=\widehat{I}_\lambda ({\widehat u}_\lambda+h)+\int_{\Omega}{\widehat F}_\lambda (x, {\widehat u}_\lambda+h)dx-\frac{\lambda}{q}|{\widehat u}_\lambda+h|_q^q-\frac{1}{p_s^*}|{\widehat u}_\lambda+h|_{p_s^*}^{p_s^*} \\
&\ge \widehat{I}_\lambda (\widehat{u}_\lambda)+\int_{\Omega}\int_0^{{\widehat u}_\lambda+h}\widehat f_\lambda(x, t)dt dx-\frac{\lambda}{q}|{\widehat u}_\lambda+h|_q^q-\frac{1}{p_s^*}|{\widehat u}_\lambda+h|_{p_s^*}^{p_s^*} \\
&={\widetilde I}_\lambda (\widehat{u}_\lambda).
\end{aligned}
\]
The last equality comes from the definition of $\widehat f_\lambda$ and ${\widehat u}_{\lambda}+h\in \Sigma$, so ${\widehat u}_{\lambda}$ is a local minimizer for $\widetilde I_\lambda$ with respect to $C_s^0(\overline\Omega)$ topology. Consider $g(t) = {\widetilde I}_\lambda(t{\widehat u}_\lambda)$, $g''(1)\ge 0$ hence \eqref{2209110032} holds. 

On the other hand, $C_c^{\infty}(\Omega)$ is dense in $W_0^{s,p}(\Omega)$ (see \cite[Theorem 2.6]{MRS2016}), so dose $W_0^{s,p}(\Omega)\cap C_s^0(\overline\Omega)$. Then $\widehat{u}_\lambda$ solves $(P_\lambda)$, hence we have
\begin{equation}\label{2209110035}
\|\widehat{u}_{\lambda}\|^p-\lambda|\widehat{u}_{\lambda}|_q^q-|\widehat{u}_{\lambda}|_{p_s^*}^{p_s^*}=0.
\end{equation}
Combining \eqref{2209110032} with \eqref{2209110035}, we get $(p_s^* - q)\lambda|\widehat{u}_{\lambda}|_q^q\geq (p^*_s - p)\|\widehat{u}_{\lambda}\|^p$. Using again \eqref{2209110035},
\begin{align*}
\widetilde I_{\lambda}(\widehat{u}_{\lambda})= \left(\frac{1}{p} - \frac{1}{p_s^*}\right)\|\widehat{u}_{\lambda}\|^p  - \left(\frac{1}{q} - \frac{1}{p_s^*}\right)\lambda|\widehat{u}_{\lambda}|_q^q < 0.\end{align*}
Moreover, for $0<\lambda<\Lambda$,
\[
\|\widehat{u}_{\lambda}\|^p\le C_1 |\widehat{u}_{\lambda}|_q^q\le C_2\|\widehat{u}_{\lambda}\|^q,
\]
which deduces that $\|\widehat{u}_{\lambda}\|$ is uniformly bounded.
\end{proof}

Now we are in position to prove the existence of positive solution for $(P_\Lambda)$. 

\begin{lemma}\label{2209270928}
Let $s\in (0,1)$, $p>1$, $n>sp$. For $\lambda=\Lambda$, there is a positive solution to $(P_\Lambda)$.
\end{lemma}
\begin{proof}
Let $\widehat{u}_\lambda$ be given in Lemma \ref{221229}, there exists $C>0$ such that $\|\widehat{u}_\lambda\|\le C$. Notice that $\widehat{u}_{\lambda}$ is also a critical point of $\widetilde{I}_\lambda$, hence a positive solution to $(P_\lambda)$. Let $u_\lambda$ be the minimal positive solution given in Lemma \ref{Minimal_Le}. So
\[
\|u_\lambda\|^p= \lambda|u_\lambda|_q^q+|u_\lambda|_{p_s^*}^{p_s^*} \le \lambda|\widehat{u}_{\lambda}|_q^q+|\widehat{u}_{\lambda}|_{p_s^*}^{p_s^*} =\|\widehat{u}_{\lambda}\|^p\le C^p.
\]
As $u_\lambda$ is increasing with respect to $\lambda$, then there exists $u_{\Lambda}\in W_0^{s,p}(\Omega)$ such that $u_\lambda$ weakly converges to $u_{\Lambda}$ and $u_\lambda(x)\to u_\Lambda(x)$ a.e. in $\Omega$. Consequently, $u_\Lambda$ is a nontrivial positive solution to $(P_\Lambda)$.
\end{proof}

Clearly, Lemma \ref{2209270925}, Lemma \ref{Minimal_Le} and Lemma \ref{2209270928} complete the proof of Theorem \ref{thm1.1}. 

\section{Proof of Theorem \ref{2209252009}}\label{2306231255}
Now, we consider the existence of second positive solution for $(P_\lambda)$. For convenience and without loss of generality, {\bf we assume $0\in \Omega$}. For $0<\lambda<\Lambda$, let $u_\lambda$ be the minimal solution given in Proposition \ref{Minimal_Le}, and ${\widehat u}_{\lambda}$ be given in Lemma \ref{221229}. If $u_{\lambda}\ne {\widehat u}_{\lambda}$, we get already two positive solutions of $(P_\lambda)$. Therefore, in this section, we always assume $$u_{\lambda}={\widehat u}_{\lambda}.$$
When $p\ge 2$, from Lemma \ref{2208251418} and Lemma \ref{221229}, it follows that $u_\lambda$ is a local minimizer of ${\widetilde I}_\lambda$ in $W_0^{s,p}(\Omega)$, that is, there exists $\rho>0$ such that
\begin{equation}\label{2209211647}
{\widetilde I}_{\lambda}(u)\ge {\widetilde I}_{\lambda}(u_{\lambda})\quad\text{for}~\text{any}~ \|u-u_{\lambda}\|\le \rho.
\end{equation}
In order to find a second positive solution, we will show a mountain pass geometry for ${\widetilde I}_\lambda$ around $u_\lambda$, we will choose the mountain pass paths from $u_\lambda$ to a terminal point $e$ such that $\|e\|>\rho$ and ${\widetilde I}_\lambda(e)<{\widetilde I}_{\lambda}(u_\lambda)$. We denote the set of mountain pass paths by $\Gamma_{\varepsilon,\delta}$, and the mountain pass level by $m_{\varepsilon,\delta}$ as in \eqref{2307142204} and \eqref{2209150938} respectively.  We will use the path $\gamma_{\varepsilon,\delta}$ given in \eqref{2301091202}. The following three claims will be checked:
 \begin{itemize}
\item[1.]  $\gamma_{\varepsilon,\delta}\in C([0,1], W_0^{s,p}(\Omega))$;
\item[2.] There exist $\delta\ge 2\varepsilon>0$ such that the maximum of ${\widetilde I}_{\lambda}$ along the path $\gamma_{\varepsilon,\delta}$ is strictly less than $c_{s, p}$ with $c_{s, p}$ given in \eqref{2307082228};
\item[3.] ${\widetilde I}_{\lambda}$ satisfies the Palais-Smale condition for any level $c<c_{s, p}$. 
\end{itemize}

Let $\eta_\delta$ be given in \eqref{2307082048}. We can get the following lemma.
\begin{lemma}\label{2210121045}
Let $B_{5\theta \delta}(0)\subset \Omega$ where $\theta$ is given in Lemma \ref{2209051345}. If $u\in W_0^{s,p}(\Omega)\cap L^{\infty}(\Omega)$, then $\eta_{\delta}u \to u$ in $W_0^{s,p}(\Omega)$ as $\delta\to 0$.
\end{lemma}
\begin{proof}
It is easy to see that there exists $C=C(p)>0$ such that
\begin{align*}
\|\eta_{\delta} u-u\|^p&\le C\int_{\mathbb{R}^{2n}}\frac{|1-\eta_{\delta}(x)|^p|u(x)-u(y)|^p}{|x-y|^{n+sp}}dxdy+C\int_{\mathbb{R}^{2n}}\frac{|(1-\eta_{\delta})(x)-(1-\eta_{\delta})(y)|^p|u(y)|^p}{|x-y|^{n+sp}} dxdy\\
&= K_1+K_2.
\end{align*}
By the Lebesgue's dominated convergence theorem, we have $\underset{\delta\to 0}{\lim}K_1=0$. Moreover, since $u\in L^{\infty}(\Omega)$,
\[
K_2\le C\|1-\eta_{\delta}\|^p=C\|\eta_{\delta}\|^p= C\|\eta_1\|^p\delta^{n-sp}.
\]
Therefore $\lim_{\delta \to 0}\|\eta_{\delta} u-u\|^p = 0$ as $n > sp$.
\end{proof}
By the above lemma, we get $\gamma_{\varepsilon,\delta}\in \Gamma_{\varepsilon,\delta}$. The second key observation is 
\begin{equation}\label{23010619}
\underset{u\in\gamma_{\varepsilon,\delta} ([0, 1])}{\sup}{\widetilde I}_{\lambda}(u)<c_{s, p},
\end{equation}
which means then $m_{\varepsilon,\delta}<c_{s, p}$. We begin with the following basic inequality.
\begin{lemma}\label{2209151434}
Assume that $p\in [2,\infty)$ and $\gamma\in (0, 2]$, then there exists $C=C(p, \gamma)>0$ such that
\begin{align*}
|a-b|^p\le a^p+b^p-pab^{p-1}+Ca^{\gamma}b^{p-2}
\end{align*}
for all $a, b > 0$.
\end{lemma}
\begin{proof}
Let $f(t)=|1-t|^p- 1 - t^p +pt$. We see that for $p \geq 2$, $\limsup_{t \to \infty} f(t) \leq 0$, then for any $\gamma \in (0, 2]$, $\sup_{t > 0} t^{-\gamma}f(t)  < \infty$, so we are done.
\end{proof}

Now we prove the claim \eqref{23010619}.
\begin{lemma}\label{2209212151}
Assume that $p\in [2, \infty)$, $p-1<q<p$, $n>\frac{sp(q+1)}{q+1-p}$. Let $\theta$ be given in Lemma \ref{2209051345}, $B_{5\theta \delta}(0)\subset \Omega$, and $m_{\varepsilon,\delta}$ be given in \eqref{2209150938}. Then there exist $\delta\ge 2\varepsilon>0$ such that $m_{\varepsilon,\delta}<c_{s, p}$.
\end{lemma}
\begin{proof}
By Proposition \ref{2209252041}, $u_{\lambda}\in C^{\alpha}(\overline{\Omega})$ for some $\alpha\in (0, s]$, which with Lemma \ref{2209151022} deduces that
\begin{equation}\label{2209022045}
\begin{aligned}
{\widetilde I}_{\lambda}(\eta_{\delta}u_\lambda)&=\frac1p\lVert \eta_{\delta}u_\lambda \rVert^p-\frac{\lambda}q |\eta_{\delta}u_\lambda|_q^q-\frac{1}{p_s^*}|\eta_{\delta}u_\lambda|_{p_s^*}^{p_s^*}\\
&\le \frac1p\lVert u_\lambda \rVert^p + C\left| (-\Delta)_p^su_\lambda \right|_{\infty}^{p/(p-1)}\delta^{n-sp}-\frac{\lambda}q |u_\lambda|_q^q-\frac{1}{p_s^*}|u_\lambda|_{p_s^*}^{p_s^*}+C\delta^n\\
&={\widetilde I}_{\lambda}(u_\lambda)+C\delta^{n-sp}+C\delta^n.
\end{aligned}
\end{equation}
We will estimate the maximum of ${\widetilde I}_{\lambda}(\eta_\delta u_\lambda + Ru_{\varepsilon,\delta})$ with respect to $R>0$.

\medskip\noindent
{\bf Step 1.} {\it Estimate for $W_0^{s,p}(\Omega)$-norm.} We claim that if $\delta/\varepsilon\ge 2$, the following estimate holds true.
\begin{equation}\label{2209021653}
\begin{aligned}
\lVert \eta_\delta u_\lambda + R u_{\varepsilon,\delta}\rVert^p & \le \lVert \eta_\delta u_\lambda\rVert^p+{R^p}\left[S^{\frac{n}{sp}}+C\left(\frac{\varepsilon}{\delta}\right)^{(n-sp)/ (p-1)}\right]\\
&\quad -C_1R^{p-1}\varepsilon^{\frac{n-sp}{p}}+ \frac{C_2R^{p-2}}{\delta^{sp}}\varepsilon^{n-\frac{(n-sp)(p-2)}{p}}\left(\frac{\delta}{\varepsilon}\right)^{n-\frac{(n-sp)(p-2)}{p-1}}.
\end{aligned}
\end{equation}
where $C_1, C_2, C_3$ are independent of $\varepsilon,\delta, R$. Indeed,
\begin{equation}\label{2301062031}
\begin{aligned}
\lVert \eta_\delta u_\lambda + R u_{\varepsilon,\delta}\rVert^p& = \int_{\mathbb{R}^{2n}}\frac{|\eta_\delta u_\lambda(x) + R u_{\varepsilon,\delta}(x)-\eta_\delta u_\lambda (y)- R u_{\varepsilon,\delta}(y)|^p}{|x-y|^{n+sp}}dxdy\\
& \le \int_{A_1}\frac{|\eta_\delta u_\lambda(x)-\eta_\delta u_\lambda (y)|^p}{|x-y|^{n+sp}}dxdy+\int_{A_2}\frac{|R u_{\varepsilon,\delta}(x)-R u_{\varepsilon,\delta}(y)|^p}{|x-y|^{n+sp}}dxdy\\
&\quad +2\int_{A_3}{\frac{|\eta_\delta u_\lambda(x)-R u_{\varepsilon,\delta}(y)|^p}{|x-y|^{n+sp}}}dxdy\\
& =: K_1+K_2+2K_3,
\end{aligned}
\end{equation}
where $A_1=B^c_{\theta \delta}(0)\times B^c_{\theta \delta}(0)$, $A_2=B_{2\theta \delta}(0)\times B_{2\theta \delta}(0)$, $A_3=B^c_{2\theta \delta}(0)\times B_{\theta \delta}(0)$. We estimate $K_3$ by Lemma \ref{2209151434} with $\gamma=2$. For $p\in [2,\infty)$, there exists $C=C(p)>0$ such that
\begin{equation}\label{2301062035}
\begin{aligned}
K_3& \le \int_{A_3}{\frac{|\eta_\delta u_\lambda(x)|^p}{|x-y|^{n+sp}}}dxdy+\int_{A_3}{\frac{|R u_{\varepsilon,\delta}(y)|^p}{|x-y|^{n+sp}}}dxdy\\
&\quad -p\int_{A_3}{\frac{|\eta_\delta u_\lambda(x)||R u_{\varepsilon,\delta}(y)|^{p-1}}{|x-y|^{n+sp}}}dxdy+C\int_{A_3}{\frac{|\eta_\delta u_\lambda(x)|^{2}|R u_{\varepsilon,\delta}(y)|^{p-2}}{|x-y|^{n+sp}}}dxdy\\
& =: L_1+L_2-L_3+L_4.
\end{aligned}
\end{equation}
First
\begin{equation}\label{2301062026}
K_1+2L_1 = \lVert \eta_\delta u_\lambda \rVert^p\quad\text{and}\quad K_2+2L_2 = \lVert R u_{\varepsilon,\delta} \rVert^p.
\end{equation}
For any $y\in B_{\theta \delta}(0)$, $x\in B^c_{2\theta \delta}(0)$, we have $|x-y|\le |x|+\theta\delta$. By \eqref{2209151133} and \eqref{2307142307}, we obtain
\begin{align*}
\begin{aligned}
L_3&\ge pR^{p-1}\int_{B_{5\theta \delta}\backslash B_{2\theta \delta}}\frac{|\eta_\delta u_\lambda(x)|}{(|x|+\theta\delta)^{n+sp}}dx\int_{B_{\theta \delta}}u_{\varepsilon,\delta}(y)^{p-1}dy\\ & \ge \frac{CR^{p-1}}{\delta^{sp}}\int_{B_{\theta \delta}}u_{\varepsilon,\delta}(y)^{p-1}dy\\
&\ge \frac{CR^{p-1}}{\delta^{sp}}\int_{B_{ \delta}}U_{\varepsilon}(y)^{p-1}dy=\frac{CR^{p-1}\varepsilon^{n-\frac{(n-sp)(p-1)}{p}}}{\delta^{sp}}\int_{B_{ \delta/\varepsilon}}U(y)^{p-1}dy.
\end{aligned}
\end{align*}
Due to Lemma \ref{2209051345} and $\delta/\varepsilon\ge 2$, we get
\[
\begin{aligned}
\int_{B_{ \delta/\varepsilon}}U(y)^{p-1}dy&\ge C\int_1^{\delta/\varepsilon}U(r)^{p-1}r^{n-1}dr \\
&\ge  C\int_1^{\delta/\varepsilon}\frac{1}{r^{n-sp}}r^{n-1}dr=\frac{C}{sp}\left[\left(\frac{\delta}{\varepsilon}\right)^{sp}-1\right]\ge \frac{C}{2sp}\left(\frac{\delta}{\varepsilon}\right)^{sp}.
 \end{aligned}
\]
Thus,
\begin{equation}\label{2301081623}
L_3\ge \frac{CR^{p-1}\varepsilon^{n-\frac{(n-sp)(p-1)}{p}}}{\delta^{sp}}\left(\frac{\delta}{\varepsilon}\right)^{sp}=CR^{p-1}\varepsilon^{n-\frac{(n-sp)(p-1)}{p}-sp}=CR^{p-1}\varepsilon^{\frac{n-sp}{p}}.
\end{equation}
For any $x\in B^c_{2\theta \delta}(0)$, $y\in B_{\theta \delta}(0)$, we have $|x-y|\ge |x|-\theta \delta \ge \frac{|x|}{2}$. It follows from \eqref{2301082313}
that
\begin{align*}
\begin{aligned}
L_4 \le \frac{CR^{p-2}}{\delta^{sp}}\int_{B_{\theta \delta}}u_{\varepsilon,\delta}(y)^{p-2}dy
& \le  \frac{CR^{p-2}}{\delta^{sp}}\int_{B_{\theta \delta}}U_{\varepsilon}(y)^{p-2}dy\\
& =\frac{CR^{p-2}}{\delta^{sp}}\varepsilon^{n-\frac{(n-sp)(p-2)}{p}}\int_{B_{\theta \delta/\varepsilon }}U(y)^{p-2}dy.
\end{aligned}
\end{align*}
Using $U\in L^{\infty}(\mathbb{R}^n)$, Lemma \ref{2209051345} and $\delta/\varepsilon\ge 2$, we have
\[
\begin{aligned}
\int_{B_{ \delta/\varepsilon}}U(y)^{p-2}dy& \le  C\int_1^{\delta/\varepsilon}U(r)^{p-2}r^{n-1}dr+C\\
& \le  C\int_1^{\delta/\varepsilon}\frac{1}{r^{\frac{(n-sp)(p-2)}{p-1}}}r^{n-1}dr+C\le C\left(\frac{\delta}{\varepsilon}\right)^{n-\frac{(n-sp)(p-2)}{p-1}}.\\
\end{aligned}
\]
Therefore
\[
L_4\le  \frac{CR^{p-2}}{\delta^{sp}}\varepsilon^{n-\frac{(n-sp)(p-2)}{p}}\left(\frac{\delta}{\varepsilon}\right)^{n-\frac{(n-sp)(p-2)}{p-1}},
\]
which together with \eqref{2301062031}--\eqref{2301081623} and \eqref{2209021656}, implies \eqref{2209021653}.

\medskip\noindent
{\bf Step 2.} {\it Estimates for power terms.}
We claim that
\begin{equation}\label{2301081854}
|\eta_{\delta}u_{\lambda}+Ru_{\varepsilon,\delta}|_q^q\ge |\eta_{\delta}u_{\lambda}|_q^q+CR^q\varepsilon^{n-\frac{(n-sp)q}{p}} \quad \mbox{for all }\; R > 0,
\end{equation}
and
\begin{equation}\label{2301081908}
|\eta_{\delta}u_{\lambda}+Ru_{\varepsilon,\delta}|_{p_s^*}^{p_s^*}\ge |\eta_{\delta}u_{\lambda}|_{p_s^*}^{p_s^*}+{R^{p_s^*}}\left[S^{\frac{n}{sp}}-C{\left(\frac{\varepsilon}{\delta}\right)}^{\frac{n}{p-1}}\right].
\end{equation}\par
Indeed, because the supports of $\eta_{\delta}u_{\lambda}$ and $u_{\varepsilon,\delta}$ are disjoint, there holds
\begin{equation}\label{2301081851}
|\eta_{\delta}u_{\lambda}+Ru_{\varepsilon,\delta}|_q^q=|\eta_{\delta}u_{\lambda}|_q^q+|Ru_{\varepsilon,\delta}|_q^q,
\end{equation}
and
\begin{equation}\label{2301081852}
|\eta_{\delta}u_{\lambda}+Ru_{\varepsilon,\delta}|_{p_s^*}^{p_s^*}=|\eta_{\delta}u_{\lambda}|_{p_s^*}^{p_s^*}+|Ru_{\varepsilon,\delta}|_{p_s^*}^{p_s^*}.
\end{equation}
By a direct computation,
\[
|u_{\varepsilon,\delta}|_q^q=\int_{B_{\theta \delta}}u_{\varepsilon,\delta}^q(y) dy\ge \int_{B_{ \delta}}U_{\varepsilon}(y)^{q}dy=\varepsilon^{n-\frac{(n-sp)q}{p}}\int_{B_{ \delta/\varepsilon}}U(y)^{q}dy.
\]
It follows from Lemma \ref{2209051345}, $\delta/\varepsilon\ge 2$ and $n>\frac{spq}{q-p+1}$ that
\[
\begin{aligned}
\int_{B_{ \delta/\varepsilon}}U(y)^{q}dy&\ge C\int_1^{\delta/\varepsilon}U(r)^{q}r^{n-1}dr \ge  C\int_1^{\delta/\varepsilon}\frac{1}{r^{(n-sp)q/(p-1)}}r^{n-1}dr\ge C.
 \end{aligned}
\]
Thus,
\begin{equation}\label{2210180957}
|u_{\varepsilon,\delta}|_q^q\ge C\varepsilon^{n-\frac{(n-sp)q}{p}}.
\end{equation}
By \eqref{2301081851}, \eqref{2210180957}, we have that \eqref{2301081854} holds. Using \eqref{2301081852} and \eqref{2209021657}, we get \eqref{2301081908}.

\medskip\noindent
{\bf Step 3.} {\it Conclusion}.\par
Combining \eqref{2209021653}, \eqref{2301081854} with \eqref{2301081908}, we deduce that
\[
\begin{aligned}
{\widetilde I}_{\lambda}(\eta_\delta u_\lambda+Ru_{\varepsilon,\delta})\le &{\widetilde I}_{\lambda}(\eta_\delta u_\lambda)+\frac{R^p}{p}\left[S^{\frac{n}{sp}}+C\left(\frac{\varepsilon}{\delta}\right)^{(n-sp)/ (p-1)}\right]-\frac{R^{p_s^*}}{p_s^*}\left[S^{\frac{n}{sp}}-C{\left(\frac{\varepsilon}{\delta}\right)}^{\frac{n}{p-1}} \right]\\
&-C_1R^q\varepsilon^{n-\frac{(n-sp)q}{p}}-C_1R^{p-1}\varepsilon^{\frac{n-sp}{p}}+ \frac{C_2R^{p-2}}{\delta^{sp}}\varepsilon^{n-\frac{(n-sp)(p-2)}{p}}\left(\frac{\delta}{\varepsilon}\right)^{n-\frac{(n-sp)(p-2)}{p-1}}.\\
\end{aligned}
\]
Taking $\varepsilon=\delta^{k+1}$ with $k>0$ in the above inequality, then
\begin{equation}\label{2301082021}
\begin{aligned}
{\widetilde I}_{\lambda}(\eta_\delta u_\lambda+Ru_{\varepsilon,\delta})
=& {\widetilde I}_{\lambda}(\eta_\delta u_\lambda)+g_{\delta}(R),
\end{aligned}
\end{equation}
with
\[
\begin{aligned}
g_{\delta}(R)=& \frac{R^p}{p}\left(S_{s,p}^{\frac{n}{sp}}+C\delta^{k(n-sp)/ (p-1)}\right)-\frac{R^{p_s^*}}{p_s^*}\left(S_{s,p}^{\frac{n}{sp}}-C{\delta}^{\frac{kn}{p-1}} \right)\\
&-C_1R^q\delta^{(k+1)\left[n-\frac{(n-sp)q}{p}\right]}-C_1R^{p-1}\delta^{\frac{(k+1)(n-sp)}{p}}+ {C_2R^{p-2}}{\delta^{(n-sp)\left[1-\frac{(p-2)(p-k-1)}{p(p-1)}\right]}}.\\
\end{aligned}
\]
Let $R_{\delta}\in \mathbb{R}^+$ satisfy
\[
g_{\delta}(R_{\delta})=\underset{R\in \mathbb{R}^+}{\max}\, g_{\delta}(R).
\]
Clearly, there exists $\delta_0>0$ such that when $\delta_0>\delta>0$, $\{R_\delta\}_{\delta_0>\delta>0}$ is bounded,  and it has a positive lower bound $T>0$. Let 
\[
h_{\delta}(R)= \frac{R^p}{p}\left(S_{s,p}^{\frac{n}{sp}}+C\delta^{k(n-sp)/ (p-1)}\right)-\frac{R^{p_s^*}}{p_s^*}\left(S_{s,p}^{\frac{n}{sp}}-C{\delta}^{\frac{kn}{p-1}} \right).
\]
Obviously, $h_R$ is increasing on $[0, R_\delta^*]$, decreasing on 
$[R_\delta^*, \infty)$ with 
$$R_\delta^* = \left(\frac{S_{s,p}^{\frac{n}{sp}}+C\delta^{k(n-sp)/ (p-1)}}{S_{s,p}^{\frac{n}{sp}}-C{\delta}^{\frac{kn}{p-1}}}\right)^\frac{1}{p_s^*-p}.$$ Therefore
\[
\underset{R\in\mathbb{R}^+}{\max}\, h_{\delta}(R)=\frac{s}{n}\frac{\left(S_{s,p}^{\frac{n}{sp}}+C\delta^{k(n-sp)/ (p-1)}\right)^{\frac{p_s^*}{p_s^*-p}}}{\left(S_{s,p}^{\frac{n}{sp}}-C{\delta}^{\frac{kn}{p-1}}\right)^{\frac{p}{p_s^*-p}}}=\frac{s}{n}S_{s,p}^{\frac{n}{sp}}+O\left( \delta^{k(n-sp)/ (p-1)}\right).
\]
Thus,
\begin{equation}\label{2307151236}
\begin{aligned}
\underset{R\in\mathbb{R}^+}{\max} g_{\delta}(R) \le & \underset{R\in\mathbb{R}^+}{\max} h_{\delta}(R)-C_1R_{\delta}^q\delta^{(k+1)\left(n-\frac{(n-sp)q}{p}\right)}\\
&-C_1R_{\delta}^{p-1}\delta^{\frac{(k+1)(n-sp)}{p}}+ {C_2R_{\delta}^{p-2}}{\delta^{(n-sp)\left[1-\frac{(p-2)(p-k-1)}{p(p-1)}\right]}}\\
=&\frac{s}{n}S_{s,p}^{\frac{n}{sp}}+O\left( \delta^{k(n-sp)/ (p-1)}\right)-C_1R_{\delta}^q\delta^{(k+1)\left[n-\frac{(n-sp)q}{p}\right]}\\
&-C_1R_{\delta}^{p-1}\delta^{\frac{(k+1)(n-sp)}{p}}+ {C_2R_{\delta}^{p-2}}{\delta^{(n-sp)\left[1-\frac{(p-2)(p-k-1)}{p(p-1)}\right]}}.\\
\end{aligned}
\end{equation}
Putting $0<k<p-1$, it follows that
\begin{equation}\label{2307151237}
\frac{k(n-sp)}{p-1}<\min\left\{(n-sp)\left[1-\frac{(p-2)(p-k-1)}{p(p-1)}\right],\frac{(k+1)(n-sp)}{p}\right\}.
\end{equation}
Using \eqref{2307151236} and \eqref{2307151237}, we get
\begin{equation}\label{2301082023}
\underset{R\in\mathbb{R}^+}{\max} g_{\delta}(R)\le \frac{s}{n}S_{s,p}^{\frac{n}{sp}}-C_1R_{\delta}^q\delta^{(k+1)\left(n-\frac{(n-sp)q}{p}\right)}+O\left( \delta^{k(n-sp)/ (p-1)}\right).
\end{equation}
By \eqref{2209022045}, \eqref{2301082021}, \eqref{2301082023} and $R_\delta>T>0$, we obtain
\[
\begin{aligned}
{\widetilde I}_{\lambda}(\eta_\delta u_\lambda+Ru_{\varepsilon,\delta})=c_{s, p}-C\delta^{(k+1)\left[n-\frac{(n-sp)q}{p}\right]}+O\left( \delta^{k(n-sp)/ (p-1)}\right)+O(\delta^{n-sp}).
\end{aligned}
\]
By $0<k<p-1$, to reach our aim, it suffices to have 
\begin{equation}\label{2210172347}
(k+1)\left[n-\frac{(n-sp)q}{p}\right]<\frac{k(n-sp)}{p-1},
\end{equation}
and it is equivalent to
\begin{equation}\label{2210172310}
n-\frac{(n-sp)q}{p}<k\left[ \frac{n-sp}{p-1}-n+\frac{(n-sp)q}{p}\right].
\end{equation}
We want to prove that there exists $0<k<p-1$ such that \eqref{2210172310} holds. In fact \eqref{2210172310} holds true with $k=p-1$, since we have
\[
n>\frac{sp(q+1)}{q+1-p}.
\]
We get then
\begin{equation}\label{2210121954}
\underset{R\in\mathbb{R}^+}{\max}\,{\widetilde I}_{\lambda}(\eta_\delta u_\lambda+Ru_{\varepsilon,\delta})<c_{s, p}
\end{equation}
if $\delta$ is sufficiently small. Combining Lemma \ref{2210121045} with \eqref{2210121954}, we conclude \eqref{23010619} provided $\delta$ is small enough, which means $m_{\varepsilon,\delta}<c_{s, p}$.
\end{proof}
Next, we proceed to check that the functional ${\widetilde I}_{\lambda}$ satisfies the Palais-Smale condition at the level $c<c_{s, p}$.  Recall that the functional ${\widetilde I}_{\lambda}$ is said satisfying the Palais-Smale condition at a level $c\in \mathbb{R}$ (for short $(PS)_c$) if any sequence $\{u_j\}\subset W_0^{s,p}(\Omega)$ such that
\begin{equation*}\label{eq2.4}
{\widetilde I}_{\lambda}(u_j)\rightarrow c \quad \mbox{and}\quad {\widetilde I}'_{\lambda}(u_j) \rightarrow 0 \; \mbox{ in }\;  W_0^{s,p}(\Omega)^*
\end{equation*}
admits a subsequence which is convergent in $ W_0^{s,p}(\Omega)$.
\begin{proposition}\label{2209210928}
Let $s\in (0,1)$, $1<q<p$, $n>sp$ and $u_\lambda$ be the minimal positive solution to $(P_\lambda)$ in Proposition \ref{Minimal_Le}. Assume that ${\widetilde I}_{\lambda}$ has only two critical points $0$ and $u_\lambda$. Then ${\widetilde I}_{\lambda}$ satisfies the $(PS)_c$ condition for all $c<c_{s, p}$.
\end{proposition}
\begin{proof}
Let $\{u_j\}\subset W_0^{s,p}(\Omega)$ be a $(PS)_c$ sequence of ${\widetilde I}_{\lambda}$ with $c<c_{s, p}$, i.e.
\begin{equation}\label{eq3.1}
 {\widetilde I}_{\lambda}(u_j)=\frac1{p}\lVert u_j \lVert^p-\frac{\lambda}{q}|u^+_j|^q_q-\frac1{p_s^*}|u^+_j|_{p_s^*}^{p_s^*}=c+o(1)
 \end{equation}
 and
\begin{equation}\label{eq3.2}
\begin{aligned}
\langle {\widetilde I}'_{\lambda}(u_j),v\rangle=&\int_{\mathbb{R}^{2n}}{\frac{J_{u_j}(x,y)(v(x)-v(y))}{|x-y|^{n+sp}}}dxdy-\lambda \int_{\Omega}(u^+_j)^{q-1}v dx\\
&-\int_{\Omega}(u^+_j)^{p_s^*-1}v dx=o(1)\lVert v \lVert \text{~~for~all~} v \in  W_0^{s,p}(\Omega).
\end{aligned}
\end{equation}
Taking $v=u_j$ in (\ref{eq3.2}), by \eqref{eq3.1} and the Sobolev embedding $W_0^{s,p}(\Omega)\subset L^q(\Omega)$, when $j$ goes to $\infty$,
\[
\begin{split}
p_s^*c+o(1)\|u_j\|+o(1)=&p_s^*{\widetilde I}_{\lambda}(u_j)-\langle {\widetilde I}'_{\lambda}(u_j), u_j \rangle\\
=&\left(\frac{p_s^*}{p}-1\right)\|u_j\|^p-\left(\frac{p_s^*}{q}-1\right)\lambda |u_j^+|_q^q\\
\ge &\left(\frac{p_s^*}{p}-1\right)\|u_j\|^p-\left(\frac{p_s^*}{q}-1\right)\lambda C\|u_j\|^q.\\
\end{split}
\]
Hence, $\{u_j\}$ is bounded in $W_0^{s,p}(\Omega)$. Therefore there is a renamed subsequence of $\{u_j\}$, which converges weakly to some $u\in W_0^{s,p}(\Omega)$ and $u_j\rightarrow u$ a.e. in $\mathbb{R}^n$. Now we will study more precisely the behavior of weakly convergent sequence $\{u_j\}$ in several steps.

\medskip\noindent
{\bf Step 1.}  We claim that 
$$\lim_{j\to \infty}(\lVert u_j-u\lVert^p - \lVert u_j \lVert^p + \lVert u \lVert^p) = 0, \quad \mbox{and} \quad \liminf_{j\to\infty}(|u_j-u|_{p_s^*}^{p_s^*} - | u^+_j|_{p_s^*}^{p_s^*} + | u^+ |_{p_s^*}^{p_s^*}) \geq 0.$$
Consider $$\Theta_j(x,y):=\frac{u_j(x)-u_j(y)}{|x-y|^{\frac{n}{p}+s}} \quad \mbox{ and }\quad\Theta(x,y):=\frac{u(x)-u(y)}{|x-y|^{\frac{n}{p}+s}}.$$ Then
$\{\Theta_j\}$ is bounded in $L^p(\mathbb{R}^{2n})$, and $\Theta_j(x,y)\rightarrow \Theta(x,y)$ a.e. in $\mathbb{R}^{2n}$. By Brezis-Lieb's lemma (see \cite[Lemma 1.32]{Willem96}), the first claim is done. Moreover, as $|u_j(x)-u(x)|\ge |u_j^+(x)-u^+(x)|$ and $u_j^+(x) \to u^+(x)$ a.e. in $\mathbb{R}^n$, using again Brezis-Lieb's lemma,
\[
|u_j-u|_{p_s^*}^{p_s^*}\ge |u^+_j-u^+|_{p_s^*}^{p_s^*}= | u^+_j|_{p_s^*}^{p_s^*}-| u^+ |_{p_s^*}^{p_s^*}+o(1),
\]
which gives the second claim.

\medskip
\noindent
{\bf Step 2.} For any $v\in W_0^{s,p}(\Omega),$
\begin{equation}\label{2208122234}
\lim_{j\to \infty} \int_{\mathbb{R}^{2n}}{\frac{J_{u_j}(x,y)(v(x)-v(y))}{|x-y|^{n+sp}}}dxdy = \int_{\mathbb{R}^{2n}}{\frac{J_{u}(x,y)(v(x)-v(y))}{|x-y|^{n+sp}}}dxdy.
\end{equation}
Indeed, define $\varPhi_j(x,y) := |x-y|^{-\frac{n+sp}{p'}}J_{u_j}(x,y)$ and $\varPhi(x,y):= |x-y|^{-\frac{n+sp}{p'}}J_{u}(x,y)$ where $p'=\frac{p}{p-1}$. Then $\{\varPhi_j\}$ is bounded in $L^{p'}(\mathbb{R}^{2n})$, and $\varPhi_j(x,y)\rightarrow \varPhi(x,y)$ a.e. in $\mathbb{R}^{2n}$, so $\varPhi_j$ converges weakly to $\varPhi$ in $L^{p'}(\mathbb{R}^{2n})$. On the other hand, $|x-y|^{-\frac{n+sp}{p}}|v(x)-v(y)|\in L^p(\mathbb{R}^{2n})$, hence \eqref{2208122234} holds.

\medskip
\noindent
{\bf Step 3.} By Step 2, it is easy to see that $u$ is a critical point of ${\widetilde I}_{\lambda}$, so
 \begin{equation}\label{eq3.6}
\lVert u \lVert^p = \lambda |u^+|_q^q+|u^+|_{p_s^*}^{p_s^*}.
\end{equation} 
Setting $\widehat{u}_j=u_j-u$, Step 1 implies then
\begin{equation}\label{eq3.3}
\lVert \widehat{u}_j\lVert^p=\lVert u_j \lVert^p-\lVert u \lVert^p+o(1)\;\; \text{and}\;\;| \widehat{u}_j|_{p_s^*}^{p_s^*}\ge | u^+_j |_{p_s^*}^{p_s^*}-| u^+ |_{p_s^*}^{p_s^*}+o(1).
\end{equation}
Taking $v=u_j$ in \eqref{eq3.2}, since $\{u_j\}$ is bounded in $W_0^{s,p}(\Omega)$ and $u_j \rightarrow u$ in $L^q(\Omega)$, we get
\begin{equation}\label{eq3.5}
\lVert u_j \lVert^p = \lambda |u^+|_q^q+|u^+_j|_{p_s^*}^{p_s^*}+o(1).
\end{equation}
It follows from (\ref{eq3.3}), (\ref{eq3.5}) and (\ref{eq3.6}) that
\begin{equation}\label{2209052311}
\begin{aligned}
\|\widehat{u}_j\|^p=|{u}_j^+|_{p_s^*}^{p_s^*}-|{u}^+|_{p_s^*}^{p_s^*}+o(1)\le | \widehat{u}_j|_{p_s^*}^{p_s^*}+o(1)\le \frac{\lVert \widehat{u}_j \lVert^{p_s^*}}{S_{s,p}^{p_s^*/p}}+o(1),
\end{aligned}
\end{equation}
so
\begin{equation}\label{eq3.7}
\lVert \widehat{u}_j \lVert^p(S_{s,p}^{p_s^*/p}-\lVert \widehat{u}_j \lVert^{p_s^*-p}) \le o(1).
\end{equation} \\
By \eqref{eq3.3} and \eqref{2209052311}, there holds
\[
\begin{aligned}
{\widetilde I}_{\lambda}(u_j)&={\widetilde I}_{\lambda}(u)+\frac{1}{p}\|\widehat{u}_j\|^p-\frac{1}{p_s^*}|{u}_j^+|_{p_s^*}^{p_s^*}+\frac{1}{p_s^*}|u^+|_{p_s^*}^{p_s^*}+o(1)\\
&={\widetilde I}_{\lambda}(u)+\frac{1}{p}\|\widehat{u}_j\|^p-\frac{1}{p_s^*}\|\widehat{u}_j\|^p+o(1)\\
&={\widetilde I}_{\lambda}(u)+\frac{s}{n}\|\widehat{u}_j\|^p+o(1).\\
\end{aligned}
\]
Hence
\[
{\widetilde I}_{\lambda}(u)+\frac{s}{n}\underset{j \rightarrow \infty}{\text{limsup}} \lVert \widehat{u}_j\lVert^p= c< c_{s, p}.
\]
As we assume that ${\widetilde I}_{\lambda}$ has only two critical points $0$ and $u_\lambda$, it follows that either $u=0$ or $u=u_\lambda$. By Lemma \ref{221229}, we know that ${\widetilde I}_{\lambda}(u_{\lambda})<0$. Hence,
\begin{equation}\label{eq3.8}
\underset{j \rightarrow \infty}{\text{limsup}} \lVert \widehat{u}_j\lVert^p< S_{s,p}^{\frac{n}{sp}}.
\end{equation}
Using (\ref{eq3.7}) and (\ref{eq3.8}), we get $\lVert \widehat{u}_j\lVert \rightarrow 0$, which means that $\{u_j\}$ has a convergent subsequence.
\end{proof}
To get the main result, we will apply Ghoussoub-Preiss' generalized mountain pass theorem \cite[Theorem (1)]{GP1989}.
\begin{theorem}\label{2307242024}
Let $X$ be a Banach space, and $\varphi$ be a $C^1$ functional on $X$. Taking $u,v \in X$ and consider 
\[
c=\underset{\gamma\in \Gamma}{\inf}\Big[\underset{0\le t\le 1}{\max}\varphi(\gamma(t))\Big]
\] 
where $\Gamma=\{\gamma\in C([0, 1], X): \gamma(0)=u, \gamma(1)=v\}$. Assume that $F$ is a closed subset of $X$ such that for any $\gamma\in \Gamma$, one has $\gamma([0, 1])\cap  \{x\in F: \varphi(x)\ge c\}\neq \varnothing$. Then there exists a sequence $\{x_j\} \subset X$ satisfying
\begin{itemize}
\item[(i)] $\underset{j\to\infty}{\lim} \mbox{dist}(x_j, F)=0$;
\item[(ii)] $\underset{j\to\infty}{\lim}\varphi(x_j)=c$;
\item[(iii)] $\underset{j\to\infty}{\lim}\|\varphi'(x_j)\|_{X^*}=0$.
\end{itemize}
\end{theorem}
\begin{lemma}\label{2210121050}
Let $s\in (0,1)$, $p\ge 2$, $p-1<q<p$ and $n>\frac{sp(q+1)}{q+1-p}$. Let $\Lambda$ be given in \eqref{2307091941}, and $m_{\varepsilon,\delta}$ be the mountain pass level defined in \eqref{2209150938} where $\delta$, $\varepsilon$ are given by Lemma \ref{2209212151}. For $\lambda\in (0, \Lambda)$, if $m_{\varepsilon,\delta}\neq 0$, then problem $(P_\lambda)$ has at least two positive solutions.
\end{lemma}
\begin{proof}
We assume by contradiction that there are only two critical points $0$ and $u_\lambda$ of ${\widetilde I}_\lambda$. It follows from Proposition \ref{2209210928} that ${\widetilde I}_\lambda$ satisfies $(PS)_c$ condition for $c<c_{s, p}$. 

Let $\rho$ be given in \eqref{2209211647}. If there exists $0<\rho_0<\rho$ such that $\inf_{u\in\partial B_{\rho_0}(u_\lambda)}{\widetilde I}_\lambda(u)>{\widetilde I}_\lambda(u_\lambda)$, then $m_{\varepsilon,\delta}>{\widetilde I}_\lambda(u_\lambda)$. As Lemma \ref{2210121045} and Lemma \ref{2209212151} showed,  $m_{\varepsilon,\delta}<c_{s, p}$. Using mountain pass theorem in \cite{AR1973},  we obtain a mountain pass critical point $w_\lambda$ of ${\widetilde I}_\lambda$. As $(-\Delta)_p^sw_\lambda=\lambda (w_\lambda^+)^{q-1}+(w_\lambda^+)^{p_s^*-1}$, by Lemmas \ref{2307091452}, $w_\lambda$ is also a nonnegative solution to $(P_\lambda)$. Since $m_{\varepsilon,\delta}\neq 0$ and $m_{\varepsilon,\delta}>{\widetilde I}_\lambda(u_\lambda)$, then $w_\lambda\notin \{0, u_\lambda\}$. So problem $(P_\lambda)$ has at least two positive solutions $u_\lambda$ and $w_\lambda$.

If $m_{\varepsilon,\delta}={\widetilde I}_{\lambda}(u_\lambda)$, then for any $0<\rho_0<\rho$, we have $\inf_{u\in\partial B_{\rho_0}(u_\lambda)}{\widetilde I}_\lambda(u)={\widetilde I}_\lambda(u_\lambda)$. Applying Theorem \ref{2307242024} with 
\[
 c=m_{\varepsilon, \delta},\;\;X=W_0^{s,p}(\Omega),\;\;F=\partial B_{\rho_0}(u_\lambda),\;\; \varphi(x)=\widetilde{I}_\lambda(x),\;\; u=u_\lambda,\;\; v=e,
\]
we obtain still another critical point $w_\lambda\in \partial B_{\rho_0}(u_\lambda)$ of ${\widetilde I}_\lambda$.
\end{proof}

\begin{remark}
From the above lemma, $(P_\lambda)$ has two positive solutions for $0<\lambda<\Lambda$ whenever $m_{\varepsilon,\delta}\neq 0$. However, we cannot rule out the case that the mountain pass critical point is trivial when $m_{\varepsilon,\delta}=0$. The following lemma can tell us that $m_{\varepsilon,\delta}>0$ if $\lambda$ is sufficiently small, which means that the trivial critical point will not occur.
\end{remark}

\begin{lemma}\label{2210121051}
Let $m_{\varepsilon,\delta}$ be the mountain pass level defined in \eqref{2209150938}. There exists $\lambda^*>0$ such that $m_{\varepsilon,\delta}>0$ for $\lambda\in(0, \lambda^*)$.
\end{lemma}
\begin{proof}
By the Sobolev embedding, for any $r\in [1, p_s^*]$, there is $C>0$ such that for any $u\in W_0^{s,p}(\Omega)$,
\[
{\widetilde I}_{\lambda}(u)\ge \frac{1}{p}\|u\|^p-\frac{C\lambda}{q}\|u\|^q-\frac{C}{p_s^*}\|u\|^{p_s^*}, \quad \forall\; u\in W_0^{s,p}(\Omega).
\]
Therefore, there exists $\lambda^*>0$ such that if $\lambda\in (0, \lambda^*)$, there are $\rho_0>0$ and $c>0$ satisfying
\begin{equation}\label{2210062140}
{\widetilde I}_\lambda(u)\ge c, \quad \mbox{for all $\|u\|=\rho_0$}.
\end{equation}
We claim that if $\lambda\in (0, \lambda^*)$, then $\|u_\lambda\|<\rho_0$. Indeed, let $g_\lambda(t)={\widetilde I}_{\lambda}(tu_\lambda)$, then
\begin{equation*}
g'_\lambda(t)=t^{p-1} \|u_\lambda\|_p^p-\lambda t^{q-1}|u_\lambda|_q^q-t^{p_s^*-1}|u_\lambda|_{p_s^*}^{p_s^*}=t^{q-1}f_\lambda(t)
\end{equation*}
where
$$f_\lambda(t):=t^{p-q}\|u_\lambda\|_p^p-{\lambda}|u_\lambda|_q^q-t^{p_s^*-q}|u_\lambda|_{p_s^*}^{p_s^*}.$$
Clearly, $f_\lambda$ has a unique maximal point $t_{\max} > 0$ such that
$f_\lambda$ is increasing on the interval $[0, t_{\max}]$ and decreasing on $[t_{\max}, \infty)$. Moreover, we can conclude $f_\lambda(t_{\max})>0$, otherwise $g_{\lambda}$ is nonincreasing, hence $g_{\lambda}(t)\le 0$ for all $t\ge 0$, which contradicts \eqref{2210062140}. Therefore, there are only two positive critical points $t_1$, $t_2$ of $g_\lambda$ satisfying $0<t_1<t_{\max}<t_2<\infty$, and $g_\lambda$ is decreasing on the intervals $(0, t_1)$ and $(t_2, \infty)$, and increasing on the interval $(t_1, t_2)$. Since $u_\lambda$ is a local minimizer of ${\widetilde I}_\lambda$, so $t_1=1$. This implies $g_\lambda(t)<0$ for $t\in (0,1]$ as ${\widetilde I}_\lambda(u_\lambda) < 0$. Hence $\|u_\lambda\|<\rho_0$. If we set $t_0$ in \eqref{23010700} large such that $\|e\|>\rho_0$, then for any $\gamma\in \Gamma_{\varepsilon,\delta}$, we have $\gamma([0, 1])\cap \partial B_{\rho_0}(0)\neq \varnothing$. As a consequence, $m_{\varepsilon,\delta}\ge c>0$.
\end{proof}
Thanks to Lemmas \ref{2210121050}, \ref{2210121051}, we complete the proof of Theorem \ref{2209252009}.

\section{Proof of Theorem \ref{2209252055}}\label{2306231256}
In this section, we will prove the existence of infinitely many solutions of $(P_\lambda)$ for $\lambda > 0$ small. Consider
\[
\mathcal{F}=\{A\subset W_0^{s,p}(\Omega)\backslash\{0\}: u\in A\Rightarrow -u\in A\},
\]
and
\[
\mathcal{A}_{j,\, r}=\{A\in \mathcal{F} : A ~\text{is}~\text{compact}, A\subset B_r(0),\, {\rm ind}(A)\ge j\}.
\]
Here ${\rm ind}(A)$ denotes the $\mathbb{Z}_2$-genus of $A$, namely, the least integer $k$ such that there exists odd functional $\phi\in C(W_0^{s,p}(\Omega), \mathbb{R}^k)$ satisfying $\phi(u)\neq 0$ for all $u\in A$. By \cite{AR1973}, the $\mathbb{Z}_2$-genus possesses the following properties:
\begin{itemize}
\item[(i)] Definiteness: ${\rm ind}(A)=0$ if and only if $A=\varnothing$;
\item[(ii)] Monotonicity: If there is an odd continuous map from $A$ to $B$ (in particular, if $A\subset B$) for $A, B\in \mathcal{F}$,  then ${\rm ind}(A)\le {\rm ind}(B)$;
\item[(iii)] Subadditivity: If $A$ and $B$ are closed set in $\mathcal{F}$, then ${\rm ind}(A\cup B)\le {\rm ind}(A)+{\rm ind}(B)$.
\end{itemize}\par
Define
\begin{equation}\label{23051421}
b_{j}:=\underset{A\in \mathcal{A}_{j, r}}{\inf}\, \underset{u\in A}{\max}\,I_\lambda(u).
\end{equation}
We are going to prove that $b_j$ is finite, a critical value of $I_\lambda$, and $b_j\to 0^-$ as $j\to \infty$.
\begin{lemma}\label{2209202147}
There exists $\lambda^{**}>0$ such that for all $\lambda\in(0, \lambda^{**}]$ there are $r, a>0$ such that
\begin{itemize}
\item[(i)] $I_\lambda(u)\ge a$ for all $\|u\|=r$;
\item[(ii)] $I_\lambda$ is bounded from below in $B_r(0) \subset W_0^{s, p}(\Omega)$;
\item[(iii)] $I_\lambda$ satisfies (PS) condition in $B_r(0)$.
\end{itemize}
\end{lemma}
\begin{proof}
Assume that $\{u_j\}\subset B_r(0)$ is a (PS) sequence, that is, $\{I_\lambda(u_j)\}$ is bounded in $\mathbb{R}$ and $I'_\lambda(u_j)\to 0$ in $W_0^{s,p}(\Omega)^*$. Since $\{u_j\}$ is bounded in $W_0^{s,p}(\Omega)$, there is a subsequence, denoted still by $\{u_j\}$, which converges weakly to some $u\in \overline B_r(0)$ and $ u_j(x)\rightarrow u(x)$ a.e. in $\mathbb{R}^n$. Arguing as Proposition \ref{2209210928}, we have
\[
\|u_j-u\|^p= | u_j-u|_{p_s^*}^{p_s^*}+o(1)\le \frac{\lVert u_j-u \lVert^{p_s^*}}{S_{s,p}^{p_s^*/p}}+o(1),
\]
which implies that
\[
\text{either}\quad S_{s,p}^{\frac{n}{sp^2}}\le \underset{j\to \infty}{\lim\inf}\lVert u_j-u\lVert\quad\text{or}\quad \underset{j\to\infty}{\lim}\|u_j-u\|=0.
\]
Therefore if we fix $0 < 2r<S_{s,p}^{\frac{n}{sp^2}}$, then $u_j$ must converge to $u$ in $W_0^{s,p}(\Omega)$.

By the Sobolev embedding $W_0^{s,p}(\Omega)\subset L^m(\Omega)$ for $m\in [1, p_s^*]$, there exists $C>0$ such that for any $u\in W_0^{s,p}(\Omega)$,
\[
I_{\lambda}(u)\ge \frac{1}{p}\|u\|^p-\frac{C\lambda}{q}\|u\|^q-\frac{C}{p_s^*}\|u\|^{p_s^*},
\]
which concludes that there exists $\lambda^{**}>0$ such that if $\lambda\in (0, \lambda^{**}]$, there are $\frac12 S_{s,p}^{\frac{n}{sp^2}}>r>0$ and $a>0$ such that $I_\lambda(u)\ge a$
for $\|u\|=r$. Obviously, $I_\lambda$ is bounded from below in $B_r(0)$.
\end{proof}

By Lemma \ref{2209202147}, the deformation lemma will hold for $I_\lambda$ restricted in $B_r(0)$ for some $r>0$. If $b_j\in (-\infty, 0)$ is not a critical value of $I_\lambda$, it follows from \cite[Lemma 3.1]{Willem96} that there exist $\varepsilon\in (0, -b_j)$ and a homotopy mapping, odd in $u$
$$\eta : [0, 1]\times I_\lambda^{b_j+\varepsilon}\to  I_\lambda^{b_j+\varepsilon}$$
such that $\eta(0, \cdot)$ is the identity map of $I_\lambda^{b_j+\varepsilon}$ and $\eta(1, I_\lambda^{b_j+\varepsilon})\subset I_\lambda^{b_j-\varepsilon}$. Here  for any $c\in \mathbb{R}$, $I_\lambda^{c}$ means the sublevel set $\{u\in B_r(0): I_\lambda(u)\le c\}$. According to the definition of $b_j$, there exists $A\in \mathcal{A}_{j, r}$ such that $A\subset I_\lambda^{b_j+\varepsilon}$. However, by means of the monotonicity of $\mathbb{Z}_2$-genus, we have
$${\rm ind}(\eta(1, A))\ge {\rm ind}(A)\ge j.$$ However, $\eta(1, A)\subset I_\lambda^{b_j-\varepsilon}$, which contradicts the definition of $b_j$. Hence, the hypothesis was wrong, $b_j$ is a critical value of $I_\lambda$.

Next, we aim to prove $b_j\to 0$ as $j\to \infty$. For that, we will consider suitable subsets of ${\mathcal A}_{j, r}$. 
Since $W_0^{s,p}(\Omega)$ is reflexive and separable, by \cite[Corollary 3.27]{Brezis2010}, $W_0^{s,p}(\Omega)^*$ is also reflexive and separable. Therefore, there exists a sequence $\{f_j\} \subset W_0^{s,p}(\Omega)^*$ such that 
$\{f_j\}$ are linearly independent and ${\rm span}\{f_j, j \ge 1\}$ is dense in $W_0^{s,p}(\Omega)^*$. Denote $F_j = {\rm span}\{f_k, 1\le k \le j\}$ for any $j \ge 1$. 

Let $g_1 = f_1$, let $v_1\in W_0^{s,p}(\Omega)$ satisfy $g_1(v_1) =1$. Clearly, there is $g_2\in F_2\backslash F_1$ such that $g_2(v_1)=0$, and there exists $v_2\in W_0^{s,p}(\Omega)$ satisfying $g_2(v_2)=1$. By induction, we get two sequences $g_j\in F_j\backslash F_{j-1}$, $v_j\in W_0^{s,p}(\Omega)$ such that
\[
\forall\; j \ge 2, \quad g_j(v_k)=0 \;\;\mbox{for }1\le k\le j-1\;\; \mbox{and}\;\; g_j(v_j)=1.
\]
Define 
\[
E_j = {\rm span}\{v_k, 1\le k \le j\},\;\;E_j^{\perp}=\underset{1 \le k\le j}{\cap} {\rm Ker}(g_k) = \underset{f \in F_j}{\cap} {\rm Ker}(f), \quad \forall\; j \ge 1.
\]
Clearly, $\{v_j\}$ are linearly independent, hence ${\rm dim}(E_j) = j$ for all $j \ge 1$. It's not difficult to see that for any $j \ge 1$, there holds $W_0^{s,p}(\Omega) = E_j \oplus E_j^{\perp}$. In other words, for any $w \in W_0^{s,p}(\Omega)$, there is a unique $w_j \in E_j$ such that $g_k(w_j) = g_k(w)$ for all $1 \le k \le j$, so $w_j^\perp = w - w_j \in E_j^{\perp}$. We denote the projection from $W_0^{s,p}(\Omega)$ into $E_j$ parallel to $E_j^\perp$ by $P_j$, i.e. $P_j(w) = w_j$. An easy but very useful observation is

\begin{lemma}\label{2306292329}
Let $\{u_j\} \subset W_0^{s,p}$ be a bounded sequence such that $u_j\in E_j^{\perp}$ for any $j$, then $u_j\to 0$ weakly in $W_0^{s,p}(\Omega)$.
\end{lemma}
\begin{proof}
Let $f\in \cup_j F_j$. From the definition of $E_j^{\perp}$, $f(u_j) = 0$ for $j$ large enough, so $\lim_{j\to \infty} f(u_j) = 0$. The same conclusion holds for any $f \in W_0^{s,p}(\Omega)^*$, since $\cup_j F_j$ is dense in $W_0^{s,p}(\Omega)^*$ and $\{u_j\}$ is bounded.
\end{proof}

\begin{proof}[Proof of Theorem \ref{2209252055}]
Let $\lambda^{**},\,r$ be as in Lemma \ref{2209202147} and $\lambda\in (0, \lambda^{**})$. Let $A_j = E_j\cap \partial B_1(0)$, clearly $A_j \in {\mathcal A}_{j, r}$ since ${\rm ind}(A_j) = j$ and $A_j$ is compact. Moreover, as ${\rm dim}(E_j) < \infty$ there is $C_j > 0$ such that for any $v \in E_j$, $\|u\|\le C_j|u|_q$. Consequently, for any $u\in A_j$ and any $t>0$,
\[
\begin{aligned}
I_\lambda(tu)\le \frac{t^p}{p}\|u\|^p-\frac{\lambda t^q}{q}|u|_q^q\le \frac{t^p}{p}\|u\|^p-\frac{C\lambda t^q}{q}\|u\|^q.
\end{aligned}
\]
There exists then $\varepsilon \in (0, r)$ satisfying 
$$\varepsilon A_j\subset B_r(0)\quad \mbox{and}\quad \underset{u\in \varepsilon A_j}{\max}I_\lambda(u)<0.$$
It means that $b_{j}$ is finite and negative. It follows that there exists a critical point $u_{j}\in B_r(0)$ of $I_{\lambda}$ with $I_\lambda(u_{j})=b_{j}<0$.  Let
\begin{equation}\label{2307072358}
\widetilde{b}_j=\underset{A\in \mathcal{A}_{j, r}}{\inf}\, \underset{u\in A\cap E_{j-1}^{\perp}}{\sup}\,I_\lambda(u).
\end{equation}
Note that $\widetilde{b}_j$ is well defined because $A\cap E_{j-1}^{\perp} \ne \varnothing$ for any $A\in \mathcal{A}_{j, r}$. In fact, if $P_{j-1}u\ne 0$ for all $u\in A$, it follows from the property of $\mathbb{Z}_2$-genus that ${\rm ind}(A)\le {\rm ind}(P_{j-1}(A))\le j-1$, which is a contradiction. Obviously $\widetilde{b}_j\le b_j$. 

Now we claim $b_{j}\to 0^-$ as $j\to \infty$. Suppose the contrary: $b_j\le \alpha < 0$ for all $j\in \NN^+$. By the definition of $\widetilde{b}_j$, there is a sequence $\{u_j\}$ such that
$$u_j\in E_{j-1}^{\perp}\cap B_r(0) \text{~and~} |I_{\lambda}(u_j)-\widetilde{b}_j|<\frac{1}{j}.$$
By Lemma \ref{2306292329}, $u_j\to 0$ weakly in $W_0^{s,p}(\Omega)$, hence $u_j\to 0$ in $L^m(\Omega)$ for all $m\in [1, p_s^*)$. As $\widetilde b_j\le \alpha$, we have 
\begin{equation}\label{2210181116}
\underset{j\to \infty}{\lim\sup}\,I_{\lambda}(u_j)\le \alpha<0.
\end{equation}
On the other hand,
\[
\begin{aligned}
I_{\lambda}(u_j) = \frac1p\|u_j\|^p-\frac{1}{p_s^*}|u_j|_{p_s^*}^{p_s^*}+o(1) &\ge \frac1p\|u_j\|^p-\frac{S_{s,p}^{-p_s^*/p}}{p_s^*}\|u_j\|^{p_s^*}+o(1)\\
&=\|u_j\|^p\left(\frac1p- \frac{S_{s,p}^{-p_s^*/p}}{p_s^*}\|u_j\|^{p_s^*-p}  \right)+o(1).
\end{aligned}
\]
By $\frac12 S_{s,p}^{\frac{n}{sp^2}}>r>0$, there holds
$$\frac1p- \frac{S_{s,p}^{-p_s^*/p}}{p_s^*}\|u_j\|^{p_s^*-p}  \ge \frac1p- \frac{S_{s,p}^{-p_s^*/p}}{p_s^*}r^{p_s^*-p}  \ge 0.$$
Hence $I_{\lambda}(u_j)\ge o(1)$, which contradicts \eqref{2210181116}. This implies $b_j\to 0^-$ and problem $(P_\lambda)$ has infinitely many solutions.
\end{proof}

\end{document}